\documentclass[12pt,a4paper]{amsart}
\usepackage[utf8]{inputenc}
\usepackage{amsmath}
\usepackage{amsfonts}
\usepackage{amssymb}
\usepackage{amsthm}
\usepackage{graphicx}
\usepackage{helvet}
\usepackage{hyperref}
\usepackage[margin=2.9cm]{geometry}

\newtheorem{theorem}{Theorem}
\newtheorem{corollary}{Corollary}

\newtheorem{remark}{Remark}

\newtheorem{lemma}{Lemma}

\begin{document}
	
	\author{Sumit Kumar, Kummari Mallesham and Saurabh Kumar Singh }
	
	\title{ Subconvex bound for $\textrm{GL(3)} \times \textrm{GL(2)}$ $L$-functions: $\textrm{GL(3)}$-spectral aspect}
	
	\address{ Sumit Kumar \newline {\em Alfr\'ed R\'enyi Institute of Mathematics,
			Budapest, Re\'altanoda utca 13-15, 1053, 
			Hungary  \newline  Email: sumit@renyi.hu
	} }

	\address{ Kummari Mallesham \newline {\em Department of Mathematics, IIT Bombay, Powai, 
			Mumbai, India 400076; \newline  Email: mallesham@math.iitb.ac.in
	} }
	\address{ Saurabh Kumar Singh \newline {\em Department of Mathematics and Statistics, Indian Institute of Technology, Kanpur, India; \newline  Email: saurabs@iitk.ac.in
	} }

	\date{\today}
	
	.
	\maketitle

	\begin{abstract}
		
		Let $\phi$ be a  Hecke-Maass cusp form for $\mathrm{SL(3, \mathbb{Z})}$ with  Langlands  parameters $({\bf t}_{i})_{i=1}^{3}$ and  $f$ be a holomorphic or Hecke-Maass cusp form  for $\mathrm{SL(2,\mathbb{Z})}$. In this article, we prove  the following subconvex bound 
		$$ L\left(\phi \times f, 1/2\right) \ll_{f,\epsilon} T^{ \frac{3}{2}-\delta_\xi+\epsilon},\  \delta_\xi=\min\{\xi/4, \, (1-2\xi)/4 \}, $$
		for the central value $  L\left(\phi \times f, 1/2\right) $ in the $\mathrm{GL(3)}$-spectral aspect,  where $({\bf t}_{i})_{i=1}^{3}$ satisfies 
		$$|{\bf t}_{3} - {\bf t}_{2}| \asymp T^{1-\xi }, \quad \, {\bf t}_{i} \asymp T,  \quad  \, \, i=1,\,2,\,3,$$
		with   $\xi$   a real number such that  $0 < \xi <1/2$. 
		
		
	\end{abstract}
	
	\tableofcontents
	
	\section{Introduction}
	An  automorphic $L$-function of degree $d$ is given by an absolutely convergent series 
	$$L(s)=\sum_{n=1}^{\infty} \frac{a_{n}}{n^{s}}$$
	for $s=\sigma + i t$ and $\sigma >1$. It has an Euler product over   primes $p$ involving the  reciprocals of degree $d$ polynomials in $p^{-s}$. One  can  complete this $L$-function by multiplying $d$ many gamma factors, and it can be extended meromorphically to the  whole of $\mathbb{C}$ with at most two simple poles at $s=0,1$. Moreover, it  satisfies a functional equation relating the value of  the $L$-function at $s$ with the `dual' $L$-function at  $1-s$.  One of the most interesting and challenging problems in analytic number theory is to estimate  $L(s)$ on the critical line $\Re s =1/2$.  Using the  Phragm\'en-Lindel\"{o}f convexity principle along with the functional equation one  can obtain 
	 $$L(1/2+it) \ll_{\epsilon} C(L)^{1/4 + \epsilon},$$ known as the convexity or trivial bound,  where $C(L)$ is  the   `analytic conductor' of the $L$-function   consisting  of various  parameters ($t$,  level,  spectral parameters etc.) associated to the underlying form.  In many application (see \cite{pmp}), one needs to improve upon the convexity bound. More precisely,  getting  a bound of the form 
	\begin{align}\label{subcon}
		L(1/2+it) \ll_{\epsilon} C(L)^{1/4 -\delta + \epsilon},
	\end{align}
	for some $\delta>0$,  is known as the  subconvexity problem.  Usually, one is   interested in  getting such  bounds  with respect to some family, i.e., varing only  one of the parameters among $t$,  level and  spectral parameters at a time.  The expected size for $L(1/2+it)$, known as the famous  Lindel\"{o}f hypothesis (LH), is  $L(1/2+it) \ll_{\epsilon}  C(L)^{\epsilon},$ which follows from  the   Grand Riemann hypothesis.  LH is still quite far from  reach using existing methods. As such getting subconvex bounds is itself a challanging task. Moreover,   getting subconvexity bounds  becomes even more difficult for higher degree $L$-functions. \\

	Subconvexity bounds for degree one $L$-functions were obtained by Weyl ($t$-aspect) \cite{Tich}  and Burgess \cite{Burg} (level aspect). For degree two, such bounds were   obtained by  Good \cite{gl3gl2ood} ($t$-aspect),   Iwaniec \cite{iwa} (spectral aspect), Duke-Friedlander-Iwaniec   \cite{DFI}, \cite{DFI-2}, \cite{DFI-2.1} (level aspect) and  Michel-Venkatesh  \cite{akii-michel} (uniform in all aspects).  \\

	   Let $\phi$ be a normalised Hecke-Maass cusp form for $SL(3, \mathbb{Z})$ with  spectral parameters $ ({\bf t}_{1}, {\bf t}_{2}, {\bf t}_{3})$. The  standard  $L$-function associated to $\phi$ is given by 
	$$L(\phi,s) = \sum_{n=1}^{\infty} \frac{A(1,n)}{n^{s}}, \ \ \ \Re s>1, $$
	where $A(1,n)$ denote the normalised  Fourier  coefficients of $\phi$.   The first subconvex bound  for  degree three $L$-functions was proved by Li \cite{Li} in the $t$-aspect. She considered $\phi$ to be a self-dual form.  This result was generalised by Munshi  \cite{munshi1} to any $GL(3)$ form (not necessarily self-dual) using the delta method.  In the level aspect (twist by a $GL(1)$ form),  Munshi   \cite{annals}  proved subconvexity for $GL(3)$ $L$-functions using the  $GL(2)$ delta method.     In the spectral aspect,  Blomer-Buttcane \cite{Bl-Butt}, using the amplification method and the  $GL(3)$-Kuznetsov trace formula, obtained the following  subconvexity bound  $$L(\phi, 1/2) \ll T^{3/4-\delta}, \ \ 0 < \delta < 1/120000,$$
	 where the spectral parameters are in  `generic' position, that is,
	 	\begin{align}\label{generic}
	 {\bf t}_{i} \asymp T,  \quad	{\bf t}_{i} - {\bf t}_{j} \asymp  T, \quad \,   \, \, i, j=1,2,3,
	 \end{align}
	 
	for a large parameter $T \geq 1$ (see Subsection \ref{notation} for notations).  Their method does not cover the case when   two  parameters, ${\bf t}_{2}$ and ${\bf t}_{3}$, say,  are near to each other, i.e., ${\bf t}_{3}- {\bf t}_{2}= o(T)$, as  the ``spectral measure" drops out in this situation. The aim of this article is to consider    forms having spectral parameters in non-generic position
		\begin{align}\label{non-generic}
	{\bf t}_{i} \asymp T, \quad 	{\bf t}_{3} - {\bf t}_{2} \asymp  T^{1-\xi } , \, \, i=1,2,3,
	\end{align}
	for some $0<\xi<1$ and $T>0$,  a large parameter.  In fact, we prove a more general result (see Theorem \ref{maintheorem}).  Subconvexity  problem for  ``genuine" degree $d \geq 4$  $L$-functions (except degree 4, 6 and 8 Rankin-Selberg $L$-functions) is still an open problem\footnote{ After  the announcement of our article on arxiv, Nelson \cite{Nel-2021} proved subconvexity for any degree $d$ $L$-functions in the $t$ and some cases of spectral aspect.}. \\
	
 Degree six $GL(3)\times GL(2)$ Rankin-Selberg  $L$-functions are  particularly interesting   due to their connection with the  Quantum Unique Ergodicity. Moreover, there   is  a structural advantage which makes them  more suitable  for  analytic techniques.   Given    a Hecke-Maass cusp form $\phi$ for $SL(3,\mathbb{Z})$ and a holomorphic Hecke cusp form $f$ for $SL(2,\mathbb{Z})$, the Rankin-Selberg    $L$-series  associated to $\phi$ and $f$  is defined as 
	$$L(\phi \times f, s) = \, \mathop{\sum \sum}_{n,r \geq 1} \frac{A(r,n) \, \lambda_{f}(n)}{(n r^2)^{s}}, \, \, \Re(s) >1,$$
	where $A(r,n)$ are the  normalised  Fourier  coefficients of  $\phi$ and $ \lambda_{f}(n)$ are the  normalised  Fourier  coefficients of  $f$.  The first subconvex bound for such $L$-functions was proved by Li \cite{Li} in the $GL(2)$-spectral aspect. The non-negativity of $L(\phi \times f, 1/2)$ is a crucial input in her method due to which  she needed to assume  $\phi$ to be a self-dual form. Using  Li's approach, Blomer \cite{Blo-2012} obtained level aspect (twist by quadratic character) subconvex bounds. This approach was further explored by  Khan \cite{Kha-2015} to prove the $GL(2)$  level aspect subconvexity conditionally.  A drawback in  the previous approach  is the difficultly  to adapt it for any $GL(3)$  form.   Recently,  Munshi \cite{munshi12}, using his delta method,   proved  subconvexity  for $L(\phi \times f,1/2+it)$ in the $t$-aspect for any $GL(3)$ form $\phi$. His method is insensitive to the self-duality of $\phi$.   Using similar ideas,   Sharma (\cite{prahlad}) obtained  a subconvexity bound for the central values $L(\phi \times f \times \chi, 1/2)$ in the twist aspect, where $\chi$ is a non-trivial character modulo $p$, a prime, thus generalised the result of Blomer \cite{Blo-2012} to any $GL(3)$ forms and non-trivial character.   Lin, Michel and  Sawin  \cite{Mic-Linn} further  extended the result of Sharma \cite{prahlad} to  `generic' trace functions. \\
	
In the  present article, we explore Munshi's approach  \cite{munshi12} further to   obtain  subconvexity for    $L(\phi \times f, 1/2) $ in the  $GL(3)$-spectral aspect, i.e., varying $\phi$ in non-generic position (see \eqref{non-generic}).    
	\begin{theorem} \label{maintheorem}
		Let $\phi$ be a Hecke-Maass cusp form for  $SL(3, \mathbb{Z})$ with  spectral parameters $({\bf t}_{1},{\bf t}_{2},{\bf t}_{3})$   satisfying \eqref{non-generic}  and  $f$ be a Holomorphic Hecke cusp form for $SL(2,\mathbb{Z})$. Then, for  $0 < \xi <1/2$,  we have
		
		$$L(\phi \times f, {1}/{2}) \ll_{f,\epsilon} \max \{ T^{\frac{3}{2}-\frac{\xi}{4}+\epsilon} , T^{\frac{3}{2}-\frac{1-2 \xi}{4}+\epsilon} \},$$
		for any $\epsilon>0$.
	\end{theorem}
	
	\begin{remark}
		Our method  can also be adapted  for $GL(2)$ Eisenstein series.  Thus we also get subconvexity for $L(\phi,1/2)$   when the  spectral parameters  satisfy  \eqref{non-generic} with $0<\xi<1/2$. Hence we complement Blomer-Buttcane's  result  \cite{Bl-Butt}. 
	\end{remark}
	\begin{remark}
		Using the same approach, along with the `mass transfer trick' (or amplification trick),  one can also  obtain  subconvexity whenever the spectral parameters are in generic position (see \eqref{generic})  under the assumption that the $GL(3)$ coefficients are `large' in $L^2$-sense.   
	\end{remark}

	\subsection{Main inputs in  the proofs } We use  the  delta  method approch.  This approach has an advantage of introducing extra `harmonics' in  exponential  sums  in which we seek some cancellations.  More specifically,   to prove our result,   we need to show some cancellations in  the   following sum
	$$\sum_{n \sim N} A(1,n) \, \lambda_{f}(n),$$
	where $N\sim T^{3}$.  
	As a first step, we use  the delta  method to seprate the oscillatory factors $A(1,n)$ and $\lambda_{f}(n)$, and while separating them we introduce new harmonics in the above sum which helps  in  lowering  the `conductor'. In fact,  we rewrite the original sum  as
	
	$$\frac{1}{K} \int W \left(\frac{v}{K} \right) \mathop{\sum \sum}_{n,m \sim T^{3}} A(1,n) n^{-i({\bf t}_{3}+v)} \, \lambda_{f}(m) \, m^{i({\bf t}_{3}+v)} \delta(n-m) \, \mathrm{d}v,$$
	where $K$ is a parameter with $T^{\epsilon }<K<T^{1-\epsilon}$ and $W$ is a smooth bump function with $\int W=1$. Note that, while separating the oscilllations, we have introduced a $v$-integral and a factor $(m/n)^{i{\bf t}_{3}}$. The $v$-integral helps in reducing the size of the equation $n=m$. Indeed, it is negligibly small unless  $n-m \ll T^{3}/K$. Now, using the delta method, we take advantange of the  `smaller  modulus'.  The factor $(m/n)^{i{\bf t}_{3}}$ plays a crucial role in  lowering the conductor while applying the $GL(3)$ Voronoi formula. This is a crucial input in the paper.
	
	\subsection{Sketch of the proof of Theorem \ref{maintheorem}}
	In this subsection, we discuss rough ideas  behind the proof of Theorem \ref{maintheorem}. Details are given from  Section \ref{circlemethod123} onwards. Using the  approximate functional equation, it boils down  to   getting non-trivial  cancellations  in the following sum (see Lemma \ref{AFE})
	$$S = \sum_{n \sim T^{3}} \, A(1,n) \, \lambda_{f}(n).$$
	In other words, we need to show $S \ll N^{1-\delta}$ for some $\delta>0$ ($S\ll N$ is the trivial bound).
	Note that $n$ is of size  $T^3$, which is the square root of the analytic conductor $T^6$ of the $L$-values $L(1/2,\phi \times f)$. 
	
	As a first step we apply the  delta method, by which we separate  the oscillatory factors  $A(1,n)$ and $\lambda_{f}(n)$ along with the conductor lowering trick. To this end, we rewrite $S$ as follows
	$$ S = \frac{1}{K} \int W \left(\frac{v}{K} \right) \mathop{\sum \sum}_{n,m \sim T^{3}} A(1,n) n^{-i({\bf t}_{3}+v)} \, \lambda_{f}(m) \, m^{i({\bf t}_{3}+v)} \delta(n-m) \, \mathrm{d}v,$$
	where $T^{\epsilon} < K < T^{1-\epsilon}$ is a parameter $K\asymp T^{1-\xi}$  and $\int W(v)\mathrm{d}v =1$.  Observe that the $v$-integral reduces the size of the equation $n-m$ by $K$ (thus the $v$-integral helps in  reducing the size of the modulus $Q$, which comes up while  applying the DFI delta method).  Note  the presence of  the term $(n/m)^{-i{\bf t}_{3}}$ which we introduced artificially. Indeed it is a crucial input in the paper.  In fact, it  serves  to ``kill" one of the spectral parameter ${\bf t}_{3}$ while applying the $GL(3)$ Voronoi  and hence it helps in  lowering  the conductor. On applying the DFI  delta method to  $\delta(n-m)$,  we arrive at
	\begin{align*}
		\frac{1}{Q^2 K} \int_{K}^{2K} & \sum_{q \sim Q} \,  \sideset{}{^\star}{\sum}_{a \, {\rm mod} \, q} \, \sum_{n \sim T^{3}} A(1,n) n^{-i({\bf t}_{3}+v)} e\left(\frac{an}{q}\right) \\
		& \times \sum_{m \sim T^{3}} \, \lambda_{f}(m) \, m^{i({\bf t}_{3}+v)} e\left(\frac{-am}{q}\right) \, \mathrm{d}v,
	\end{align*}
	where $Q=\sqrt{T^{3}/K}$. At this stage, on  estimating the above sum  trivially, we see that 
	$S \ll T^{6+\epsilon}.$ Thus   we need to save $T^{3}$ (and a little more) in the following expression 
	$$\int_{K}^{2K} \sum_{q \sim Q} \,  \sideset{}{^\star}{\sum}_{a \, {\rm mod} \, q} \, \sum_{n \sim T^{3}} A(1,n) n^{-i({\bf t}_3+v)} e\left(\frac{an}{q}\right) \, \sum_{m \sim T^{3}} \, \lambda_{f}(m) \, m^{i({\bf t}_3+v)}  e\left(\frac{-am}{q}\right) \,\mathrm{d}v,$$
	where by the term ``  saving", we mean $ \mathrm{initial \ bound} / \mathrm{final \ bound}$.
	In the next step we apply the  Voronoi summation formulas to  the sum over   $m$ and $n$. On applying the  $GL(3)$ Voronoi summation formula to the $n$-sum, the length of the dual sum becomes $n^{*} \sim Q^3 K^2 T /T^{3}$ and we save $T^{3}/\sqrt{Q^3 K^2 T}$ in this step.
	
	The  $GL(2)$ Voronoi  converts the $m$-sum into a dual sum of the length $m^{*} \sim Q^2 T^{2}/T^{3}$ and this step  gives a saving of the size $T^{3}/ \sqrt{Q^2 T^{2}} $. 
	
	We also save $\sqrt{Q}$ in the $a$-sum and $\sqrt{K}$ in the $v$-integral. So far we have saved 
	$$ \frac{T^{3}} {\sqrt{Q^3 K^2 T}} \frac{T^{3}}{\sqrt{Q^2 T^2}} \sqrt{Q} \sqrt{K} = T^{3} \frac{K^{1/2}}{T^{3/2}}.$$
	Therefore, we need to save $T^{3/2}/K^{1/2}$ and a little more in the following transformed sum 
	$$ \sum_{q \sim Q} \sum_{n^{*} \sim \frac{Q^3 K^2 T}{T^{3}}} \, A(1,n) \sum_{m^{*} \sim \frac{Q^2 T^{2}}{T^{3}}} \, \lambda_{f}(m) \, \mathfrak{C} \, \mathfrak{I},$$
	where $\mathfrak{I}$ is an integral transform which oscillates like $n^{i K}$ with respect to $n$, and the character sum is given by 
	$$\mathfrak{C} = \sideset{}{^{*}}{\sum}_{a \, \rm mod \, q} S\left(\bar{a},n;q\right) \, e\left(\frac{\bar{a} m}{q}\right) \mapsto q e\left(-\frac{\bar{m} n}{q}\right).$$
	In the next step we apply the Cauchy inequality in the $n$-sum to get rid of the coefficients $A(1,n)$ and we arrive at the following expression 
	\begin{align}\label{after-cauchy}
		\sum_{n^{*} \sim \frac{Q^3 K^2 T}{T^{3}}} \Big \vert \sum_{q \sim Q} \sum_{m^{*} \sim \frac{Q^2 T^{2}}{T^{3}}} \lambda_{f}(m) \, e\left( - \frac{\bar{m} n}{q}\right) \, \mathfrak{I}\Big \vert^{2},
	\end{align}
	where we seek to save $T^{3}/K$ plus little more. Opening the absolute value square we apply the Poisson summation formula to the $n$-sum. In the zero frequency we save $Q^{3} T^{2}/T^{3}$ which is satisfactory if 
	$$\frac{Q^3 T^{2}}{T^{3}} > \frac{T^{3}}{K}$$
	or equivalently $K < T$. In the non-zero frequency we save 
	$$\frac{Q^3 K^2 T}{T^{3} \sqrt{K}}$$ 
	which is sufficient if 
	$$\frac{Q^3 K^2 T}{T^{3} \sqrt{K}} > \frac{T^{3}}{K}$$
	i.e., $K > T^{1/2}$. Thus we get the following  restriction on the choice  of $K$
	$$T^{1/2} < K < T.$$
	Therefore, we have a room to choose $K$ optimally and hence we get the result.

	\subsection{Notations}\label{notation}
	For any complex number $z$ we set $e(z):= e^{2 \pi i z}$. By $A \asymp B$ we mean that $ T^{-\epsilon} B \leq A \leq  T^{\epsilon}B$. By $A \approx B$ we mean $c_{1} A \leq B \leq c_{2} A$ for some positive real $c_{1},c_{2}$. By $A \sim B$ we mean that $B \leq A \leq 2B$. By the notation $X \ll Y$ we mean that for any $\epsilon >0$, there is a constant $c >0$ such that $|X| \leq c Y $. The implied constants may depend on the cusp form $ f$ and $\epsilon$. At various places $\epsilon>0$ may be different.

	\section{Preliminaries}
	\subsection{Back ground on the Maass forms for $GL(3)$} \label{gl3 maass form}
	Let $\phi$ be a Maass form of type $(\nu_{1}, \nu_{2})$ for $SL(3, \mathbb{Z})$. The  spectral parameters $({\bf t}_{1},{\bf t}_{2},{\bf t}_{3})$ of $\phi$ are given by
	$${\bf t}_{1} = - \nu_{1} - 2 \nu_{2}+1, \, {\bf t}_{2} = - \nu_{1}+ \nu_{2},  \, {\bf t}_{3} = 2 \nu_{1}+ \nu_{2}-1.$$
	Note that they satisfy the following equation:
	$${\bf t}_{1}+{\bf t}_{2}+{\bf t}_{2}=0.$$
	Let $T$ be a large parameter such that 
	$${\bf t}_{i} \approx T , \, \, \mathrm{for}\ \ i=1,2,3.$$
	Thus  $\nu_i$'s are of size at most $T$. In fact, one of them has size $T$ while the other one has size at most $T$. Without loss of generality, let's assume that $\nu_2 \approx T$. Hence $\nu_1 \ll T$. Using this we conclude that 
	$${\bf t}_{1} - {\bf t}_{2} \approx {\bf t}_{1} - {\bf t}_{3} \approx T,\ \ {\bf t}_{2} - {\bf t}_{3} \ll T.$$
	By the work of Jacquet, Piatetskii-Shapiro and Shalika, we have the Fourier Whittaker expansion for the Maass form $\phi(z)$ as follows
	\begin{equation} \label{four-whi-exp}
		\phi(z) = \sum_{\gamma \in U_{2}\left(\mathbb{Z}\right) \backslash  SL(2,\mathbb{Z})} \sum_{m_{1}=1}^{\infty} \sum_{m_{2} \neq 0} \frac{A(m_{1},m_{2})}{m_{1} |m_{2}|} \, W_{J}\left(M \left( {\begin{array}{cc}
				\gamma &  \\
				& 1 \\
		\end{array} } \right)
		z, \nu, \psi_{1,1}\right),
	\end{equation}
	where $U_{2}(\mathbb{Z})$ is the group of upper triangular matrices with integer entries and ones on the diagonal,
	$W_{J}\left(z,\nu,\psi_{1,1}\right)$ is the Jacquet-Whittaker function, and $M=\textbf{diag} \left(m_{1}|m_{2}|,m_{1},1\right)$ (cf. Goldfeld \cite{gold}).
	
	For $\psi(x) \in C_{c}^{\infty}(0, \infty)$, we denote $\tilde{\psi}(s) = \int_{0}^{\infty} \psi(x) x^{s-1} \mathrm{d}x$ to be its Mellin transform. 
	For $\ell =0$ and $1$, we set
	\begin{equation}
		\gamma_{\ell}(s) :=  \frac{\pi^{-3s-\frac{3}{2}}}{2} \, \prod_{i=1}^{3} \frac{\Gamma\left(\frac{1+s+{\bf t}_{i}+ \ell}{2}\right)}{\Gamma\left(\frac{-s-{\bf t}_{i}+ \ell}{2}\right)}. 
	\end{equation}
	
	Set $\gamma_{\pm}(s)= \gamma_{0}(s) \mp i\gamma_{1}(s)$ and let
	$$G_{\pm}(y) = \frac{1}{2 \pi i} \int_{(\sigma)} y^{-s} \, \gamma_{\pm}(s) \, \tilde{\psi}(-s) \, \mathrm{d}s $$
	with $\sigma > -1 + \max \{-\Re({\bf t}_{1}), -\Re({\bf t}_{2}), \Re({\bf t}_{3})\}$. 
	
	With the aid of the above terminology we state the $GL(3)$-Voronoi summation formula, due to Miller and Schmid \cite{miller-schmid}, in the following proposition.
	\begin{lemma} \label{gl3voronoi}
		Let $\psi (x)$ be a compactly supported smooth function on $(0,\infty)$. Let $A(m,n)$ be the $(m,n)$-th Fourier coefficient of a Maass form $\phi(z)$ for $SL(3,\mathbb{Z})$. Then we have
		\begin{align} \label{GL3-Voro}
			& \sum_{n=1}^{\infty} A(m,n) e\left(\frac{an}{q}\right) \psi(n) \\
			\nonumber & =q  \sum_{\pm} \sum_{n_{1}|qm} \sum_{n_{2}=1}^{\infty}  \frac{A(n_{2},n_{1})}{n_{1} n_{2}} S\left(m \bar{a}, \pm n_{2}; mq/n_{1}\right) \, G_{\pm} \left(\frac{n_{1}^2 n_{2}}{q^3 m}\right)
		\end{align} 
		where $(a,q)=1, \bar{a}$ is the multiplicative inverse modulo $q$ and $$S(a,b;q) = \sideset{}{^\star}{\sum}_{x \,\rm mod \, q} e\left(\frac{ax+b\bar{x}}{q}\right) $$
		is the Kloostermann sum.
	\end{lemma}

	The following lemma gives the Ramanujan bound for $A(m,n)$ on average, see \cite{gold}.
	\begin{lemma} \label{ramanubound}
		We have 
		$$\mathop{\sum \sum}_{n_{1}^{2} n_{2} \leq x} \vert A(n_{1},n_{2})\vert ^{2} \ll \, x^{1+\epsilon}.$$
	\end{lemma}

	\subsection{Back ground on the holomorphic cusp forms on $GL(2)$}
	Let $f$ be a holomorphic Hecke-eigen form for the full modular group $SL(2,\mathbb{Z})$ of weight $k \geq 12$, even integer. Then we have the Fourier expansion 
	$$f(z) = \sum_{n=1}^{\infty} \, \lambda_{f}(n) \, n^{k-1/2} \, e(nz)$$
	for $z \in \mathbb{H}$. We have Deligne's bound on Fourier coefficients $\vert \lambda_{f}(n) \vert \leq d(n)$ for all $n\geq 1$, where $d(n)$ is the divisor function. 
	
	We have the following Voronoi summation formula, see \cite[Chapter 4]{iwaniec} for example.

	\begin{lemma} \label{gl2voronoi}
		Let $g$ be a smooth, compactly supported function on $\mathbb{R}^{+}$. Then we have
		$$\sum_{n=1}^{\infty} \lambda_{f}(n) \, e\left(\frac{an}{c}\right)\, g(n) = \frac{1}{c} \sum_{n=1}^{\infty} \lambda_{f}(n) \, e\left(-\frac{d n}{c}\right)\, h(n),$$
		where $ad \equiv 1 (\rm mod \, c)$ and 
		$$h(y) = 2 \pi i^{k} \int_{0}^{\infty} g(x) \, J_{k-1} \left(\frac{4 \pi \sqrt{xy}}{c}\right) \, dx.$$
	\end{lemma}

	\subsection{The Rankin-Selberg $L$-function on $GL(3) \times GL(2)$}
	The Rankin-Selberg $L$-function on $GL(3) \times GL(2)$ is defined as
	$$L(\phi \times f, s) = \, \mathop{\sum \sum}_{n,r \geq 1} \frac{A(r,n) \, \lambda_{f}(n)}{(n r^2)^{s}}, \, \, \Re(s) >1.$$
	The completed $L$-function 
	\begin{align} \label{complete f e}
		\Lambda(\phi \times f,s):= \gamma(\phi \times f,s) \, L(\phi \times f, s).
	\end{align}
	The factor $\gamma(\phi \times f,s)$ at infinity is given as
	$$\gamma(\phi \times f,s) = \prod_{\substack{1\leq i \leq 3 \\ 1\leq j \leq 2}} \, \Gamma_{\mathbb{R}}\left(s-\mu_{f,j}-{\bf t}_{i}\right),$$ 
	where $$\Gamma_{\mathbb{R}}(s) = \pi^{-s/2} \Gamma(s/2), \mu_{f,1}= -(k-1)/2, \mu_{f,2} = -k/2.$$
	The completed $L$-function $\Lambda(\phi \times f,s)$ satisfies the functional equation, see \cite{gold},
	\begin{equation} \label{functioeq}
		\Lambda(\phi \times f,s) = \varepsilon(\phi \times f ) \, \Lambda(\bar{\phi} \times f,1-s),
	\end{equation}
	where the root number $\varepsilon(\phi \times f ) =i^{3k}$ and the form $\bar{\phi}$ is dual to $\phi$ and the $(m,n)$th coefficient for $\bar{\phi}$ is given by $A(n,m)$.
	\subsection{DFI $\delta$-method} \label{circlemethod}
	Let $\delta: \mathbb{Z} \to \{0,1\}$ be defined by
	\[
	\delta(n)= \begin{cases}
		1 \quad \text{if} \,\  n=0; \\
		0 \quad  \,  $\textrm{otherwise}$.
	\end{cases}
	\]
	The above delta symbol can be used to separate the oscillations involved in a sum. Further, we seek a Fourier expansion of $\delta(n)$. We mention here an expansion for $\delta(n)$ which is due to Duke, Friedlander and Iwaniec. Let $L\geq 1$ be a large number. For $n \in [-2L,2L]$, we have
	
	\begin{align*} 
		\delta(n)= \frac{1}{Q} \sum_{1 \leq q \leq Q} \frac{1}{q} \, \sideset{}{^\star}{\sum}_{a \, \rm mod \, q} \, e \left(\frac{na}{q}\right) \int_{\mathbb{R}} g(q,x) \,  e\left(\frac{nx}{qQ}\right) \, \mathrm{d}x,
	\end{align*}
	where  $Q=2L^{1/2}$. The $\star$ on the sum indicates that the sum over $a$ is restricted by the condition $(a,q)=1$. The function $g$ is the only part in the above formula which is not explicitly given. Nevertheless, we only need the following  properties of $g$ in our analysis.  
	\begin{align} \label{g properties}
		&g(q,x)=1+h(q,x), \quad \text{with} \ \ \  h(q,x)=O \left(\frac{Q}{q} \left(\frac{q}{Q}+|x|\right)^{B}\right),  \\
		& x^j \frac{\partial ^j}{\partial x^j}g(q,x) \ll \log Q \min \left\lbrace \frac{Q}{q}, \frac{1}{|x|}\right\rbrace, \notag  \\
		& g(q,x) \ll |x|^{-B} \notag.  
	\end{align}
	for any $B>1$ and $j \geq 1$. Using the third property of $g(q,x)$ we observe that the effective range of the integration in \eqref{deltasymbol}  is $[-L^{\epsilon},L^{\epsilon}]$. We record the above observations in the following lemma.
	\begin{lemma}\label{deltasymbol}
		Let $\delta$ be as above and $g$ be a function satisfying \eqref{g properties}. Let $L\geq 1$ be a large parameter. Then, for $n \in [-2L,2L]$, we have
		\begin{equation*} 
			\delta(n)= \frac{1}{Q} \sum_{1 \leq q \leq Q} \frac{1}{q} \, \sideset{}{^\star}{\sum}_{a \, \rm mod \, q} \, e \left(\frac{na}{q}\right) \int_{\mathbb{R}}V(x/Q^\epsilon) g(q,x) \,  e\left(\frac{nx}{qQ}\right) \, \mathrm{d}x+O(L^{-2020}),
		\end{equation*}
		where $Q=2L^{1/2}$ and  $V(x)$ is a smooth bump function supported in $[-2,2]$, with $V(x)=1$ for $x \in [-1,1] $ and $W^{(j)}\ll_j 1$.
	\end{lemma}
	\begin{proof}
		See  [Chapter 20, \cite{iwaniec}].
	\end{proof}
	
	\subsection{Gamma function} \label{gamma}
	We will  need  the following  Stirling asymptotic formula.

	For fixed $\sigma \in \mathbb{R}$,  $\vert \tau\vert \geq 10$ and any $M>0$ we have, 
	\begin{align} \label{stirling}
		\Gamma(\sigma + i \tau) = e^{-\frac{\pi}{2}|\tau|} \, \vert \tau \vert^{\sigma-\frac{1}{2}}\exp\left(i \tau \, \log \frac{\vert \tau \vert}{e}\right)\, g_{\sigma, M}(\tau) + O_{\sigma,M}(\tau^{-M})
	\end{align}
	where $t^{j} g_{\sigma,M}^{(j)}(t) \ll_{j,\sigma,M } 1$.
	
	
	
	%
	\subsection{Oscillatory integrals}
	
	We need bounds on the exponential integrals. In this section we collect some results which give bounds and asymptotic formulas of the exponential integrals. 
	
	\subsubsection{One dimensional exponential integrals}
	The results of this subsection are taken from \cite[Section 8]{Blomer}. 
	
	We are interested in the integrals of the form 
	$$I = \int w(t) \, e^{ih(t)} \mathrm{d}t$$
	where $w$ is a smooth function with support in $[a,b]$ and $h$ is a smooth real valued function on $[a,b]$. 
	
	The following lemma will be used to show that exponentials integrals $I$ are negligibly small in the absence of the stationary phase. 
	\begin{lemma} \label{nostationary}
		Let $Y \geq 1, X, Q, U, R >0$. And let us further assume that
		\begin{itemize}
			\item $w^{(j)}(t) \ll_{j} \frac{X}{U^{j}}$ for $j=0,1,2, \ldots$
			\item $\vert h^{\prime}(t) \vert \geq R$ and $h^{(j)}(t) \ll_{j} \frac{Y}{Q^{j}}$ for $j=2,3,\ldots$
		\end{itemize}
		Then we have
		$$I \ll_{A} (b-a)X \left(\left(\frac{QR}{\sqrt{Y}}\right)^{-A}+\left(RU \right)^{-A} \right).$$
	\end{lemma}
	
	The following lemma gives an asymptotic expression for $I$ when the stationary phase exist.
	\begin{lemma} \label{stationaryphase}
		Let $0 < \delta < 1/10, X, Y, U, Q>0, Z:= Q+X+Y+b-a+1$, and assume that
		$$ Y \geq Z^{3 \delta}, \, b-a \geq U \geq \frac{Q Z^{\frac{\delta}{2}}}{\sqrt{Y}}.$$
		Assume that $w$  satisfies  
		$$w^{(j)}(t) \ll_{j} \frac{X}{U^{j}} \, \, \, \text{for} \,\,  j=0,1,2,\ldots.$$  
		Suppose that there exists unique $t_{0} \in [a,b]$ such that $h^{\prime}(t_{0})=0$, and the function $h$ satisfies
		$$h^{\prime \prime}(t) \gg \frac{Y}{Q^2}, \, \, h^{(j)}(t) \ll_{j} \frac{Y}{Q^{j}} \, \, \, \, \text{for} \, \, j=1,2,3,\ldots.$$
		Then we have
		$$I = \frac{e^{i h(t_{0})}}{\sqrt{h^{\prime \prime}(t_{0})}} \, \sum_{n=0}^{3 \delta^{-1}A} p_{n}(t_{0}) + O_{A,\delta}\left( Z^{-A}\right), \, p_{n}(t_{0}) = \frac{\sqrt{2 \pi} e^{\pi i/4}}{n!} \left(\frac{i}{2 h^{\prime \prime}(t_{0})}\right)^{n} G^{(2n)}(t_{0})$$
		where 
		$$ G(t)=w(t) e^{i H(t)}, \text{and} \, H(t)= h(t)-h(t_{0})-\frac{1}{2} h^{\prime \prime}(t_{0})(t-t_{0})^2.$$
		Furthermore, each  $p_{n}$ is a rational function in $h^{\prime}, h^{\prime \prime}, \ldots,$ satisfying the derivative bound
		$$\frac{d^{j}}{dt_{0}^{j}} p_{n}(t_{0}) \ll_{j,n} X \left(\frac{1}{U^{j}}+ \frac{1}{Q^{j}}\right) \left( \left(\frac{U^2 Y}{Q^2}\right)^{-n} + Y^{-\frac{n}{3}}\right).$$
	\end{lemma}

	\subsubsection{Two dimensional exponential integrals}
	The result of this section is taken from \cite[Section~2.2.1, Lemma~2.6 ]{kratzel}. We record in this subsection the second derivative bound for exponential integrals in two variables. Let $f,g:[a,b] \times [c,d] \to \mathbb{R}$ be smooth functions. Then we are interested in the integral of the form
	$$\int_{a}^{b} \int_{c}^{d} g(x,y) \, e\left(f(x,y)\right) \mathrm{d}x \, \mathrm{d}y.$$


	Let $D$ be a plane domain which is contained in the rectangle $[a,b]\times [c,d]$ with $b-a \geq 1$ and $d-c \geq 1$. Let $f:[a,b]\times [c,d] \mapsto \mathbb{R}$ be any function with continuous partial derivatives of as many orders as may required. Then we want to estimate the exponential integral 
	$$\mathop{\int \int}_{D} \, e\left(f(x,y)\right) \, dx \, dy.$$
	We give bounds for the above integral in the following lemma. It requires some conditions on the domain $D$ and the function $f$. In fact, we need following conditions:
	\begin{enumerate}
		\item Any straight line parallel to any of the coordinate axes intersects $D$ in a bounded number of line segments.
		\item  Suppose that $\frac{\partial}{\partial x} f(x,y), \frac{\partial}{\partial y} f(x,y)$ are monotonic in $x$ and $y$, respectively. 
		\item Intersections of $D$ with domains of the type $$\frac{\partial}{\partial x} f(x,y), \frac{\partial}{\partial y} f(x,y) \geq c$$ or $$\frac{\partial}{\partial x} f(x,y), \frac{\partial}{\partial y} f(x,y) \leq c$$ are to satisfy condition (1).
		\item The boundary of $D$ can be divided into a bounded number of parts. In each part the curve of the boundary is given by $y =\text{constant}$ or $x= \rho(y)$, which is continuous. 
	\end{enumerate} 
	Under these assumptions,  we have the following lemma.
	
	\begin{lemma}
		Suppose that 
		$$\Lambda_{1}^2 \ll \frac{\partial^2}{\partial x^2} f(x,y) \ll \Lambda_{1}^2, \, \,  \Lambda_{2}^2 \ll \frac{\partial^2}{\partial y^2} f(x,y) \ll \Lambda_{2}^2,$$
		$$ \frac{\partial^2}{\partial x \partial y} f(x,y)  \ll \Lambda_{1} \Lambda_{2}, \, \, \frac{\partial^2}{\partial x^2} f(x,y)\frac{\partial^2}{\partial y^2} f(x,y) - \left(\frac{\partial^2}{\partial x \partial y} f(x,y) \right)^2 \gg \Lambda_{1}^2 \, \Lambda_{2}^2$$
		through out the rectangle $[a,b] \times [c,d]$. For all parts of the curve of the boundary let $y = \text{constant}$ and $x = \rho(y)$, where $\rho(y)$ is twice differentiable and $\rho^{\prime \prime}(t) \ll r$. Then 
		$$\int_{a}^{b} \int_{c}^{d} g(x,y) \, e\left(f(x,y)\right) \mathrm{d}x \, \mathrm{d}y \ll \frac{1 + \log (b-a)(d-c) + \log \Lambda_{1} + \log \Lambda_{2}}{\Lambda_{1} \Lambda_{2}} + \frac{r}{\Lambda_{2}}.$$ 
	\end{lemma}
	
	\begin{corollary} \label{twodimexpointe}
		Let $g:[a,b]\times [c,d] \mapsto \mathbb{R}$ be a compactly supported smooth function and the  support of $g$ lies in $[a,b] \times [c,d]$ . Then we have
		\begin{align*}
			&\int_{a}^{b} \int_{c}^{d} g(x,y) \, e\left(f(x,y)\right) \mathrm{d}x \, \mathrm{d}y  \\
			&\ll \left(\frac{1 + \log (b-a)(d-c) + |\log \Lambda_{1}| + |\log \Lambda_{2}|}{\Lambda_{1} \Lambda_{2}} + \frac{r}{\Lambda_{2}} \right) \text{Var}(g),
		\end{align*}
		where $\text{var}(g)$ is the total variation of $g$ which  is defined as 
		$$\text{var}(g):= \int_{a}^{b} \int_{c}^{d} \left| \frac{\partial^2}{\partial x \partial y} g(x,y)\right|  \mathrm{d}x \mathrm{d}y.$$ 
	\end{corollary}
	\begin{proof}
		To deduce  the corollary, we apply integration by parts once in each variable of the integral. 
	\end{proof}

	%

	\section{Set Up} \label{circlemethod123}
	Let $\phi$ be a Hecke-Maass cusp form, as defined in Subsection \ref{gl3 maass form},  whose Langlands parameters satisfy the following conditions:
	\begin{align}\label{parameters condition}
		{\bf t}_{1} - {\bf t}_{2} \asymp {\bf t}_{1} - {\bf t}_{3} \asymp  T,\ \ {\bf t}_{2} - {\bf t}_{3} \asymp  T^{1-\xi}, \ \  {\bf t}_{j} \asymp  T \quad \, \, j=1,2,3,
	\end{align}
	%
	for $0<\xi<1$.  Let $f$ be a Holomorphic/Maass Hecke-cusp form for the modular group $\mathrm{SL(2,\mathbb{Z})}$. 
	To estimate the central value $L(\phi \times f,1/2)$, we first express it as a weighted Dirichlet series.  
	\begin{lemma}\label{AFE}
		We have
		\begin{align} \label{centrallvalues}
			L(\phi \times f,1/2) \ll_{\epsilon} T^{\epsilon}\sup_{r \leq T^{(3+\epsilon)/2}} \sup_{  Nr^2\leq {T^{3+\epsilon}}} \frac{\left|S_r(N)\right|}{N^{1/2}} + T^{-2020},
		\end{align}
		where $S_r(N)$ is an exponential  sum 
		\begin{align} \label{s(n)-sum}
			S_r(N): = \mathop{\sum }_{n=1}^{\infty} A(r,n) \lambda_{f}(n)  W \left(\frac{n}{N}\right),
		\end{align}
		for some smooth function $W$ supported in $[1,2]$ and satisfying $W^{(j)}(x) \ll_{j} 1$ with $\int W(x) dx =1$. 
	\end{lemma} 
	\begin{proof}
		The proof is standard.  Indeed, it follows from   the approximate functional equation (see \cite{iwaniec}, Theorem 5.3) of $	L(\phi \times f,s)$. 
	\end{proof}
	Note that, upon estimating \eqref{centrallvalues} trivially, we get $$L\left( \phi \times f, {1}/{2}\right) \ll N^{\frac{1}{2}+\epsilon} \ll T^{3/2+\epsilon}.$$
	Thus to get subconvexity  we   need   to get  non-trivial  cancellations in  $S_{r}(N)$. 
	\section{An application of the  delta method} \label{delta symbol inserting}
	As a first step,  we  separate the  oscillatory terms $A(r,n)$ and $\lambda_{f}(n)$  using the delta method. But this step  alone does not suffice for our purpose. We also  reduce the size of the equation $n=m$ detected by the delta symbol by using the  `conductor lowering trick' introduced by   Munshi  in \cite{munshi1}. To this end, we write $S_r(N)$ as 
	\begin{align} \label{s(n)-sum11}
		&\mathop{\sum \sum}_{\substack{m,n=1 \\ n=m}}^{\infty} \, A(r,n) \, \lambda_f(m)W \left(\frac{n}{N}\right) U \left(\frac{m}{N}\right)  \\
		&= \frac{1}{K} \, \int_{\mathbb{R}} W \left(\frac{v}{K}\right) \mathop{\sum \sum}_{\substack{m,n=1 \\ n=m}}^{\infty} \, A(r,n) \, \lambda_f(m)  \left(\frac{m}{n}\right)^{i \left( \Im{\bf t}_{3}+v \right)} \,  W \left(\frac{n}{N}\right) \, U \left(\frac{m}{N}\right) \mathrm{d}v\notag,
	\end{align}
	where $U$ is a smooth function supported in $[1/2,5/2]$ with $U(x)=1$ for $x \in [1,2]$ and $U^{(j)}(x) \ll_{j} 1$, and $K$ is a parameter such that $T^\epsilon<K<T$ . Later we will choose  $K=T^{1-\xi}$.   Notice the presence of the artificial factor  $(m/n)^{i\Im {\bf t}_3}$. A priori, its role is not clear. However it is a crucial input in the paper. Indeed it is useful in getting better dual length  from   the $\mathrm{GL(3)}$ Voronoi formula. As such it  kills the parameter $\Im {\bf t}_3$ in the gamma factor (see Subsection  \ref{Gl3 applying} for more details).  
	By repeated integration by parts, we observe that the $v$-integral 
	$$\frac{1}{K} \int_{\mathbb{R}} \, W\left(\frac{v}{k}\right) \left(\frac{m}{n}\right)^{i v} \, \mathrm{d}v$$ is negligibly small if $\vert n-m\vert \gg N T^{\epsilon}/K$.     We now apply the delta method due to Duke-Friedlander-Iwaniec  to detect the equation $n-m=0$ with the modulus 
	\begin{equation} \label{qvalue}
		Q = \sqrt{\frac{N}{K}} T^{\epsilon}.
	\end{equation}
	On applying Lemma \ref{deltasymbol} to \eqref{s(n)-sum11}, we get that 
	\begin{align} \label{mainsum}
		S_{r}(N) &= \frac{1}{QK} \int_{\mathbb{R}} \, V(x) \,\int_{\mathbb{R}} W \left(\frac{v}{K} \right) \sum_{1\leq q \leq Q} \frac{g(q,x)}{q} \,  \sideset{}{^\star}{\sum}_{a \, {\rm mod}\, q}   \\
		& \nonumber \times \sum_{n=1}^{\infty} \, A(r,n) \, n^{-i( \Im {\bf t}_{3}+v)} \,  e\left(\frac{na}{q}\right) e\left(\frac{nx}{q Q}\right) W\left(\frac{n}{N}\right) \\
		&\nonumber \times  \sum_{m=1}^{\infty} \lambda_{f}(m)  \, m^{i( \Im {\bf t}_{3}+v)} \, e \left(-\frac{ma}{q}\right) e\left(-\frac{mx}{qQ} \right)U\left(\frac{m}{N}\right)   \ dv \  dx, 
	\end{align}
	where $V(x)$ is a bump function with  support in $[-T^{\epsilon},T^{\epsilon}]$.
	Note that the trivial estimation of $S_{r}(N)$  at this stage yields 
	$S_{r}(N) \ll N^{2+\epsilon}.$
	%

	\section{Applications of  $\mathrm{GL(3)}$ and $\mathrm{GL(2)}$ Voronoi Summation formulae}
	In this section, we apply the  $\mathrm{GL(3)}$ and $\mathrm{GL(2)}$ Voronoi summation formulae to the $n$-sum and  the $m$-sum in \eqref{mainsum} respectively.
	
	\subsection{$\textrm{GL(2)}$ Voronoi formula}\label{gl2 for thm 1}
	In this subsection, we apply the  $\mathrm{GL(2)}$ Voronoi summation formula  to the sum over $m$ in \eqref{mainsum} 
	\begin{align}\mathrm{S_2}:= \sum_{m=1}^{\infty} \lambda_{f}(m)  \, m^{i(\Im{\bf t}_{3}+v)} \, e \left(-\frac{ma}{q}\right) e\left(-\frac{mx}{qQ} \right)U\left(\frac{m}{N}\right).
	\end{align}
	
	\begin{lemma} \label{gl2voronoi123}
		We have 
		\begin{align*}
			\mathrm{S_2}
			= \frac{N^{\frac{3}{4}+i(\Im {\bf t}_{3}+v)}}{q^{1/2}}\sum_{\pm} \sum_{m \leq \widetilde{M}} \frac{\lambda_{f}(m)}{m^{1/4}} \, e \left(\frac{\overline{a} m}{q}\right) I_{\pm}(m,x,q) + O(T^{-2020}),
		\end{align*}
		where 
		\begin{equation} \label{iprimeinte}
			I_{\pm}(m,x,q) = \int_{0}^{\infty} U_\pm(y_{2}) y_{2}^{i(\Im{\bf t}_{3}+v)} e \left(-\frac{xNy_{2}}{qQ} \pm \frac{2 \sqrt{y_{2}Nm}}{q}\right) dy_{2},
		\end{equation}
		and $\widetilde{M}$ is  defined in \eqref{mtilde}.
	\end{lemma}
	\begin{proof}
		On  applying  the $\mathrm{GL(2)}$-Voronoi summation formula  (see lemma  \ref{gl2voronoi}) with $g(y)= y^{i(\Im{\bf t}_{3}+v)} e(-yx/qQ) U(y/N)$ to $\mathrm{S_2}$, we get that 
		\begin{equation} \label{cuttingm}
			\mathrm{S_2}= \frac{2 \pi i^k}{q} \sum_{m=1}^{\infty} \lambda_{f}(m) \,e \left(\frac{\overline{a} m}{q}\right) I_{1}(m,x,q),
		\end{equation}	
		where
		\begin{align*}
			I_1(m,x,q)&= \int_{0}^{\infty} y_{2}^{i(\Im{\bf t}_{3}+v)} e \left(\frac{-x y_{2}}{qQ}\right) U\left(\frac{y_{2}}{N} \right) J_{k-1} \left(\frac{4 \pi \sqrt{y_{2}m}}{q}\right) \mathrm{d} y_{2} \\
			&=N^{1+i(\Im{\bf t}_{3}+v)}\int_{0}^{\infty} y_{2}^{i(\Im{\bf t}_{3}+v)} e \left(\frac{-Nx y_{2}}{qQ}\right) U(y_2) J_{k-1} \left(\frac{4 \pi \sqrt{Ny_{2}m}}{q}\right) \mathrm{d} y_{2}.
		\end{align*} 
		Here $k$ is the weight of the  form $f$.  Using  standard properties of the  Bessel functions, we can express $J_{k-1}$ as
		$$J_{k-1}\left(4 \pi z \right)=e(2z) W_{k-1}\left(2z\right) + e\left(-2z\right) \, \overline{W}_{k-1}\left(2z\right),$$
		where $W_{k-1}$ is a smooth function satisfying 
		$$z^{j} W_{k-1}^{(j)}(z) \ll_{j,k} \frac{1}{\sqrt{z}}, \ \ j \geq 0,$$
		for $z \gg 1$. 
		On extracting the oscillations  of $J_{k-1}$, i.e., writing $W_{k-1}$ as 
		$$W_{k-1}(z)=\frac{1}{\sqrt{z}}W_{k-1}^{\prime}(z), \ \ \mathrm{with} \ \ W_{k-1}^{\prime (j)}(z) \ll 1/z^{j}, $$
		we see that the integral $I_1(...)$  can be  expressed as
		\begin{align} \label{gl2 integral}
			\sum_{\pm}\frac{\sqrt{q}N^{1+i(\Im{\bf t}_{3}+v)}}{(Nm)^{1/4}}\int_{0}^{\infty} y_{2}^{i(\Im{\bf t}_{3}+v)} U_{\pm}(y_2)  e \left(-\frac{xNy_{2}}{qQ} \pm \frac{2 \sqrt{y_{2}Nm}}{q}\right)dy_{2},
		\end{align}
		where $U_{\pm}(y_2)$ is a new smooth weight function of the form $$U_+(y_2)={U(y_2)W^{\prime}_{k-1}(2\sqrt{Ny_2m}/q)}/{y_2^{1/4}}, \  U_-(y_2)={U(y_2)\overline{W^{\prime}}_{k-1}(2\sqrt{Ny_2m}/q)}/{y_2^{1/4}}.$$
		By repeated integration by parts, we observe that the above integral  is negligibly small if 
		\begin{equation} \label{mtilde}
			m \geq \max \left\lbrace\frac{T^2 q^2}{N}, K \right\rbrace T^{\epsilon}:=\widetilde{M}.
		\end{equation}
		Finally plugging \eqref{gl2 integral} into \eqref{cuttingm} we get the lemma.
	\end{proof}

	\begin{remark} \label{sizeofnx}
		Let $q \sim C$ and $m \not \asymp \frac{T^2 C^2}{N}$. Then $N|x|/CQ \gg T^{1-\epsilon}$,  otherwise the integral $I_{\pm}(m,x,q)$ is negligibly small.
	\end{remark}

	\subsection{ $\textrm{GL(3)}$ Voronoi formula} \label{Gl3 applying}
	Let 
	\begin{align}\label{s3}
		\mathrm{S_3}:= \sum_{n=1}^{\infty} \, A(r,n) \, n^{-i(\Im{\bf t}_{3}+v)} \,  e\left(\frac{na}{q}\right) e\left(\frac{nx}{q Q}\right) W\left(\frac{n}{N}\right).
	\end{align}
	On applying the $\mathrm{GL(3)}$ Voronoi formula (Lemma \ref{gl3voronoi}) with $$\psi(y)= y^{-i(\Im{\bf t}_{3}+v)} \,  e\left(\frac{yx}{q Q}\right) W\left(\frac{y}{N}\right),$$
	we see that 
	\begin{equation} \label{spivalue}
		\mathrm{S_3}=q \sum_{\pm} \sum_{n_{1}|r q} \sum_{n_{2}=1}^{\infty} \,  \frac{ A(n_{1},n_{2})}{n_{1}n_{2}} \, S\left(r \overline{a}, \pm n_{2};qr/n_{1}\right) G_{\pm} \left(n_2^\star\right),
	\end{equation}
	where $n_2^\star = {n_{1}^{2} n_{2}}/({q ^{3}r})$ and 
	\begin{align} \label{int gl3}
		G_{\pm}( n_2^\star ) = \frac{1}{2 \pi i} \int_{(\sigma)} (n_2^\star)^{-s} \, \gamma_{\pm}(s) \, \tilde{g}(-s) \, \mathrm{d}s.
	\end{align}
	Let  $G_{\pm}(n_2^\star)= G_{0}(n_2^\star) \mp G_{1}(n_2^\star),$
	with 
	$$G_{\ell}(n_2^\star) = \frac{1}{2 \pi i} \int_{(\sigma)} (n_2^\star)^{-s} \, \gamma_{\ell}(s) \, \widetilde{g}(-s) \, \mathrm{d}s, \ \ \   \ell =0,1. $$
	Let's consider  $\tilde{g}(-s)$, which is given by 
	\begin{align}\label{mellin of psi}
		\notag \widetilde{g}(-s)&=\int_{0}^{\infty} V\left(\frac{y_{1}}{N}\right) y_{1}^{-i(  \Im {\bf t}_{3}+v)} e \left(\frac{y_{1}x}{qQ}\right) y_{1}^{-s-1}  dy_{1} \\
		&=N^{-i( \Im {\bf t}_{3}+v)-s}\int_{0}^{\infty} V \left(y_{1}\right) y_{1}^{-i( \Im {\bf t}_{3}+v)} e \left(\frac{y_{1}N x}{qQ}\right) y_{1}^{-s-1}  \, \mathrm{d} y_{1}. 
	\end{align} 
	By repeated integration by parts, we see  that  the above integral is negligibly small if 
	\begin{align}\label{tau+alpha}
		\vert \tau +  \Im {\bf t}_{3} \vert \gg    T^\epsilon\max \Big \{ K,  \frac{N}{qQ}\Big\},
	\end{align}
	where $\tau =\Im s$.  
	Next we analyse  $G_{\pm}\left( n_2^\star \right)$,
	%
	which is given by  
	\begin{align*}
		G_{\pm}(n_2^\star) &= \frac{1}{2 \pi i} \int_{(\sigma)} (n_2^\star)^{-s} \, \gamma_{\pm}(s) \, \widetilde{g}(-s) \, \mathrm{d}s \\ 
		&= \frac{1}{2 \pi } \int_{\mathbb{R}} (n_2^\star)^{-\sigma-i\tau } \, \gamma_{\pm}(\sigma+i\tau) \, \widetilde{g}(-\sigma-i\tau) \, \mathrm{d}\tau. 
	\end{align*}
	Let $N_0:=T^\epsilon N/(qQ) $.  We  decompose  the above  integral as follows:
	$$	G_{\pm}(n_2^\star) =	G_{\pm}^{(1)}(n_2^\star) +	G_{\pm}^{(2)}(n_2^\star) +	G_{\pm}^{(3)}(n_2^\star),$$
	where
	$$	G_{\pm}^{(j)}(n_2^\star)=  \frac{1}{2 \pi } \int_{A_j} (n_2^\star)^{-\sigma-i\tau } \, \gamma_{\pm}(\sigma+i\tau) \, \widetilde{g}(-\sigma-i\tau) \, \mathrm{d}\tau, \ \ j=1,2,3,$$
	with
	\begin{align*}
		&A_1=\{\tau_1 \in \mathbb{R} | \,  N_0^{1-\epsilon} \leq  |\tau_1+{\Im {\bf t}_3}| \leq N_0^{1+\epsilon}  \}, \\  
		& A_2=\{\tau_1 \in \mathbb{R} | \,  |\tau_1+{\Im {\bf t}_3}| < N_0^{1-\epsilon}  \}, \  A_3=\{\tau_1 \in \mathbb{R} | \,  |\tau_1+{\Im {\bf t}_3}| > N_0^{1+\epsilon}  \}. 
	\end{align*}
	Note that $G_{\pm}^{(3)}(n_2^\star)$  is negligibly small due to \eqref{tau+alpha} as $N_0 \gg T^\epsilon\{K, N|x|/(qQ) \} $.  Next we analyse    $G_{\pm}^{(1)}(n_2^\star)$. 
	
	%
	%
	\noindent
	\subsubsection{ Analysis  of   $G_{\pm}^{(1)}(n_2^\star)$.}
	In this case, we observe that 
	\begin{align}\label{tau+1}
		|\tau_1+\Im {\bf t}_2|=|  \tau_1+\Im{\bf t}_3-\Im{\bf t}_3+\Im{\bf t}_2| \asymp N_0,
	\end{align}
	unless $ N_0 \asymp  |\Im{\bf t}_2-\Im {\bf t}_{3}| \asymp T^{1-\xi} $. We choose $K = T^{1-\xi}$. Thus if   $N_0 \asymp K$,   we might have $|\tau_1+\Im {\bf t}_{2}| \ll K$. We also have 
	\begin{align}\label{tau+2}
		|\tau_1+\Im{\bf t}_1|=|\tau_1+\Im{\bf t}_3-\Im {\bf t}_3+\Im {\bf t}_1| \asymp  (N_0+T),
	\end{align}
	unless  $N_0 \asymp | \Im{\bf t}_1-\Im{\bf t}_3| \asymp T $ in which case, we might have $ |\tau+\Im {\bf t}_1| \ll N_0+T$.  Let 
	$$G_{\pm}^{(1)}(n_2^\star)=G_{\pm}^{(1,1)}(n_2^\star)+G_{\pm}^{(1,2)}(n_2^\star),$$
	where $$G_{\pm}^{(1,j)}(n_2^\star)= \frac{1}{2 \pi } \int_{A_{(1,j)}}  (n_2^\star)^{-\sigma-i\tau } \, \gamma_{\pm}(\sigma+i\tau) \, \widetilde{g}(-\sigma-i\tau) \, \mathrm{d}\tau, \ \ j=1,2,  $$
	with \begin{align}\label{A11}
		A_{(1,1)}:=\{\tau_1\in A_1| \ T^\epsilon \leq | \tau_1+\Im {\bf t}_2| \ll N_0,\ T^\epsilon \leq | \tau_1+\Im {\bf t}_1| \ll N_0+T  \}
	\end{align}
	and $A_{(1,2)}:=A_1\backslash A_{(1,1)}$.  We will  now analyse these integrals using Stirling's formula   \eqref{stirling}. Let's consider $G_{\pm}^{(1,1)}(n_2^\star)$, which is given as 
	$$G_{\pm}^{(1,1)}(n_2^\star)= G_{0}^{(1,1)}(n_2^\star)+ G_{1}^{(1,1)}(n_2^\star),$$
	where  
	$$G_{\ell}^{(1,1)}(n_2^\star)= \int_{A_{(1,1)} } (n_2^\star)^{-\sigma-i\tau } \, \gamma_{\ell}(\sigma+i\tau) \, \widetilde{g}(-\sigma-i\tau) \, \mathrm{d}\tau,  \ \ \ \ell=0,1.$$
	We will  further focus on  $G_{0}^{(1,1)}(n_2^\star)$ as $G_{1}^{(1,1)}(n_2^\star)$ can  be  analysed similarly. 
	Consider 
	$$\gamma_{0}(\sigma+i\tau)= \frac{\pi^{-3(\sigma+i\tau)-\frac{3}{2}}}{2} \, \prod_{j=1}^{3} \frac{\Gamma\left(\frac{1+\sigma+i\tau+ {\bf t}_j}{2}\right)}{\Gamma\left(\frac{-\sigma-i\tau-  {\bf t}_j}{2}\right)}.$$
	Let's  assume for simplicity that  $\Re {\bf t}_j=0$, $1\leq j \leq 3$. Indeed this condition holds under the generalised Ramanujan conjecture.
	For notational convenience, we continue writing  ${\bf t}_j$ in place of $\Im {\bf t}_j$ (abusing notation). Thus $\gamma_0(\sigma+i\tau)$ is given as 
	$$\gamma_{0}(\sigma+i\tau)= \frac{\pi^{-3s-\frac{3}{2}}}{2} \, \prod_{j=1}^{3} \frac{\Gamma\left(\frac{1+\sigma+i\tau+ i{\bf t}_j}{2}\right)}{\Gamma\left(\frac{-\sigma-i\tau- i {\bf t}_j}{2}\right)},$$
	On applying  Stirling's formula  \eqref{stirling} to $\gamma_0(\sigma+i\tau)$, we see that 
	\begin{align}\label{gamma bound}
		\gamma_{0}(\sigma+i\tau) \ll  \prod_{j=1}^{3} e^{-\frac{\pi}{2}\,  |\tau+{\bf t}_{j}|}|\tau+{\bf t}_j|^{1/2+\sigma}.
	\end{align}
	%
	Thus on plugging it  in place of $\gamma_0(\sigma+i \tau)$ and a corresponding expression for  $\gamma_1(\sigma+i \tau)$  in place of $\gamma_1(\sigma+i\tau)$ in  $G_{0}^{(1,1)} \left(n_2^\star\right)$ and $G_{1}^{(1,1)} \left(n_2^\star \right)$ respectively, we see that 
	\begin{align*}
		G_{\pm}^{(1,1)} \left( n_2^\star \right)\ll \int_{A_{(1,1)}} \left(\frac{ 8 \pi^3  n_{2} ^\star N}{ \prod_{j=1}^{3}  |\tau + {\bf t}_j |}\right)^{- \sigma}  \prod_{j=1}^{3} e^{-\frac{\pi}{2}\,  |\tau+{\bf t}_{j}|} \,  |\tau+{\bf t}_j |^{1/2}\,   \mathrm{d}\tau.
	\end{align*}
	Here we  used $ \widetilde{g}(-\sigma-i\tau) \ll N^{-\sigma}$ from \eqref{mellin of psi}.  Upon shifting $\sigma$ to the right towards  infinity, we see that  the above integral is negligibly small if 
	\begin{align*}
		\frac{8 \pi^3  n_{2}^\star N}{ \prod_{j=1}^{3} |\tau + {\bf t}_j |} \gg T^{\epsilon} \iff n_1^2n_2 \gg \frac{q^3 rT^{\epsilon} }{N}\prod_{j=1}^{3} |\tau + {\bf t}_j |.
	\end{align*}
	
	By  \eqref{tau+1} and   \eqref{tau+2}, we see that 
	\begin{align} \label{size of tau}
		\prod_{j=1}^{3} |\tau + \alpha_j | \ll T^\epsilon \left(T+\frac{N}{qQ}\right)\left(\frac{N}{qQ}\right)\frac{N}{qQ} \ll \left(T+\frac{N}{qQ}\right)\frac{N^2T^\epsilon}{q^2Q^2},
	\end{align}     
	Thus    $G_{\pm}^{(1,1)} \left(n_2^\star \right)$ is negligibly small unless 
	$$  n_1^2n_2  \ll  \frac{ q^3rT^\epsilon}{ N} \left(T+\frac{N}{qQ}\right)\frac{N^2}{q^2Q^2}= { rT^\epsilon} \left(Tq+\frac{N}{Q}\right)\frac{N}{Q^2}.$$
	Note that 
	\begin{align}\label{3.2}
		{ rT^\epsilon} \left(Tq+\frac{N}{Q}\right)\frac{N}{Q^2} \ll   { rT^\epsilon} \left(TQ+\frac{N}{Q}\right)\frac{N}{Q^2} \ll  { rT^\epsilon} \frac{NT}{Q}. 
	\end{align}
	In further analysis, we will work with the range   
	\begin{align}\label{ntilde12}
		n_2n_1^2  \ll  { rT^\epsilon} \frac{NT}{Q}=\frac{rT^\epsilon Q^3 K^2 T}{N} :={\widetilde{N}}.
	\end{align}
	In this range, on shifting the contour  to $\sigma =-1/2$, and using Stirling formula   \eqref{stirling},  we see that  $G_{\pm}^{(1,1)} \left(n_2^\star\right)$, upto a negligibly small error term, is given by 
	\begin{align*}
		&\frac{\left(n_2^\star N\right)^{{1}/{2}}}{N^{-i(v +{\bf t}_3 )}}  \int_{A_{(1,1)}} W_1 (\tau) e \left( h_1(\tau) \right) 
		\int_{0}^{\infty} V\left(y_{1}\right) y_{1}^{-i({\bf t}_3+v)} e \left(\frac{y_{1}N x}{qQ}\right) y_{1}^{-1/2 - i \tau}  \, \mathrm{d} y_{1}\ \mathrm{d}\tau,
	\end{align*}
	where $W_1$ is a smooth function satisfying $\tau^jW_1^{(j)}(\tau) \ll_{\epsilon,j }  T^{j\epsilon}$ and 
	\begin{align} \label{h0value}
		h_1(\tau)= -\frac{\tau}{2\pi}+\frac{1}{2\pi} \sum_{j=1}^{3} (\tau + {\bf t}_j )\log|\tau + {\bf t}_j | -\frac{\tau}{2\pi}\log \left(8\pi^3n_2^\star N \right). 
	\end{align}
	We  now  split the range of the $\tau$-integral  as follows: 
	$$ \int_{A_{(1,1)}} = \sum_{ N_0^{1-\epsilon}  \ll  \mathfrak{c} \ll N_0^{1+\epsilon}} \int_{ A_{\mathfrak{c}}}, $$
	where $A_{\mathfrak{c}}$ consists of those $\tau_1 \in \mathbb{R}$ such that 
	\begin{align}\label{ac}
		\mathfrak{c} \leq |\tau_1+{\bf t}_3| \leq \mathfrak{c}(1+1/10),\,   T^\epsilon \leq  |\tau_1+{\bf t}_1| \ll (\mathfrak{c} +T)T^\epsilon, T^\epsilon \leq  |\tau_1+{\bf t}_2| \ll \mathfrak{c}^{1+\epsilon},
	\end{align}
	and $\mathfrak{c} $ is of the form $N_0^{1-\epsilon}(1+\ell /10)$ for $0\leq \ell \ll N_0^{2\epsilon}$. 
	It means that we divide the range into the segments  $[N_0^{1-\epsilon },  (1+1/10) N_0^{1-\epsilon}),$ $[N_0^{1-\epsilon }(1+1/10),  (1+2/10) N_0^{1-\epsilon})$, and so on. Thus we arrive at the following expression of $G_{\pm}^{(1,1)} \left(n_2^\star\right)$ 
	\begin{align*}
		&\frac{\left(n_2^\star N\right)^{{1}/{2}}}{N^{-i(v +{\bf t}_3 )}}   \sum_{ \mathfrak{c} } \int_{ A_{\mathfrak{c}}} W_1 (\tau) e \left( h_1(\tau) \right) 
		\int_{0}^{\infty} V\left(y_{1}\right) y_{1}^{-i({\bf t}_3+v)} e \left(\frac{y_{1}N x}{qQ}\right) y_{1}^{-1/2 - i \tau}  \, \mathrm{d} y_{1}\ \mathrm{d}\tau.
	\end{align*}

			%
			Next we consider $G_{\pm}^{(1,2)} \left(n_2^\star\right)$,  which is given by 
			$$G_{\pm}^{(1,2)}(n_2^\star)= \int_{A_{(1,2)} } (n_2^\star)^{-\sigma-i\tau } \, \gamma_{\pm }(\sigma+i\tau) \, \widetilde{g}(-\sigma-i\tau) \, \mathrm{d}\tau,  $$
			where  $A_{(1,2)}=A_{(1,2,1)} \cup A_{(1,2,2)}$ with 
			$$A_{(1,2, \ell)} =\{\tau_1 \in A_{(1,2)}  |\,   \tau_1+{\bf t}_3| \asymp N_0,\,     |\tau_1+{\bf t}_\ell| <T^\epsilon \}, \ \ \ell=1,2.$$
			Note that $A_{(1,2, 1)}\cap A_{(1,2, 2)}=\phi$. Thus $G_{\pm}^{(1,2)}(n_2^\star)=G_{\pm}^{(1,2,1)}(n_2^\star)+G_{\pm}^{(1,2,2)}(n_2^\star),$
			where $$G_{\pm}^{(1,2,\ell)}(n_2^\star)= \int_{A_{(1,2, \ell)} } (n_2^\star)^{-\sigma-i\tau } \, \gamma_{\pm }(\sigma+i\tau) \, \widetilde{g}(-\sigma-i\tau) \, \mathrm{d}\tau, \ \ \ell=1,2.$$
			Let's   consider $G_{\pm}^{(1,2,1)}(n_2^\star)$. 
			Recall from \eqref{tau+2} that this situation can occur only if $N_0 \asymp T$. Thus we have $q\ll T^\epsilon QK/T$. 
			%
			Since $\Re(1+s+i{\bf t}_3)>0$,  the following Stirling's bound (see  \eqref{stirling} and \eqref{gamma bound}) still holds:
			\begin{align}
				\gamma_{0}(\sigma+i\tau) \ll  \prod_{j=1}^{3} e^{-\frac{\pi}{2}\,  |\tau+{\bf t}_{j}|}|\tau+{\bf t}_j|^{1/2+\sigma}.
			\end{align}
			Now  proceeding as before, we see that $G_{\pm}^{(1,2,1)}(n_2^\star)$ is 
			negligibly small if 
			\begin{align*}
				\frac{8 \pi^3 n_{1}^2 n_{2}N}{q^3 r \prod_{j=1}^{3} |\tau + {\bf t}_{j}|} \gg T^{\epsilon} \iff  n_1^2n_2 \gg    \frac{q^3 rT^{\epsilon} }{N}\prod_{j=1}^{3} |\tau + {\bf t}_{j}|.
			\end{align*}
			Using $ \prod_{j=1}^{3} | \tau + {\bf t}_{j}| \ll T^{2+\epsilon}$ and   $q\ll T^{\epsilon} QK/T$, we see that   the effective range of $n_2$ is given by 
			\begin{align}\label{n2 nongeneric}
				n_2 \ll \frac{1}{n_1^2}\frac{r Q^3 K^3 }{TN} T^{\epsilon}.
			\end{align}
			Note that the above  length is  $K/T^2$ times  the generic  $n_2$-length  \eqref{ntilde12}. On   analysing the terms corresponding to $\tau+{\bf  t}_2$ and  $\tau+{\bf t}_3$ like before, we see that we save more in this situation   at the last (as the  $n_2$-length is shorter).  Thus,  we will get better bounds for $G_{\pm}^{(1,2,1)} \left(n_2^\star\right)$. 
			On analysing  $G_{\pm}^{(1,2,2)}(n_2^\star)$ in the same way, and using $ \prod_{j=1}^{3} | \tau + {\bf t}_{j}| \ll KT^{1+\epsilon}$, $q\leq Q$, we see that $n_2$-length is $1/K$-times the generic $n_2$-length in this situation. Hence we save $K$ extra at the end.  
			
			\ 
			
			\noindent
			{\bf Case 2}: Analysis of  $G_{\pm}^{(2)}(n_2^\star)$. 
			
			In this case, we have  $ \vert \tau + {\bf t}_3\vert \ll  N_0^{1-\epsilon}$, which can happen only if  $N/(qQ) \asymp K \asymp N_0 $. On using \eqref{tau+2}, we see that  $|\tau+{\bf t}_{1}| \asymp T $, as $N_0<T$. 
			We also  note that 
			\begin{align*}
				|\tau+{\bf t}_2|=|\tau+{\bf t}_3-{\bf t}_3+{\bf t}_2| \asymp |{\bf t}_2-{\bf t}_{3}| \asymp K.
			\end{align*}
			%
			%
			%
			%
			Note that   we are in a similar situation as that  in Case 1.   Indeed the  role of ${\bf t}_{2}$ and ${\bf t}_{3}$ got interchanged here.  Now  for $T^\epsilon\leq  |\tau+{\bf t}_3| \ll K^{1-\epsilon}$, we can  analyse  the integral tranform like Case 1.    The case  $|\tau+{\bf t}_3|<T^\epsilon$ can also  be    analysed  similarly.   In further analysis, we will focus on Case 1 only.  
			We summarise this subsection in the following lemma. 
			\begin{lemma} \label{gl3voroi1233}
				Let $\mathrm{S_3}$ be the sum over $n_2$ defined in \eqref{s3}. Then  we arrive at the following expression for $\mathrm{S_3}$
				\begin{align} \label{s3 after gl3}
					\frac{N^{1/2 - i( {{\bf t}_3} +v)}}{q^{1/2} r^{1/2}}  \sum_{\pm}  \sum_{n_{1}|qr} &\sum_{n_{2} \ll {\widetilde{N}}/ {n_{1}^{2}}} \frac{ \lambda_\pi(n_1,n_2) }{ \sqrt{ n_2}}   S \left(r \overline{a}, \pm n_{2};qr/n_{1} \right) G_{\star} \left({ n_{2}^\star N}\right)  \notag\\
					&{\mathrm{lower\ order \ terms}}+O(T^{-2020}),
				\end{align}
				where the integral transform $	G_{\star} \left( n_2^\star N\right)$ is given by 
				\begin{align} \label{gl3integral}
					& \sum_{ N_0^{1-\epsilon}  \ll  \mathfrak{c} \ll N_0^{1+\epsilon}}  \int_{A_{\mathfrak{c}}} W_{1}(\tau) e \left(h_{1}(\tau) \right) \int_{0}^{\infty} V \left(y_{1}\right) y_{1}^{-i({{\bf t}_3}+v)} e \left(\frac{y_{1}N x}{qQ}\right) y_{1}^{-1/2 - i \tau}  \, \mathrm{d} y_{1}\ \mathrm{d}\tau ,
				\end{align}
				where  $A_{\mathfrak{c}}$ is the set defined  in \eqref{ac},   $W_1$ is a smooth function satisfying $x^jW_1^{(j)}(x) \ll_j T^\epsilon$ and 
				\begin{align}\label{h_1-1}
					h_1(\tau)=-\frac{\tau}{2\pi}+\frac{1}{2\pi}\sum_{j=1}^{3} (\tau + {{\bf t}_j})\log|\tau +{{\bf t}_j}| - \frac{\tau}{2\pi} \log \left(8\pi^3 n_{2}^\star N \right).
				\end{align} 
			\end{lemma}

			\subsection{$S_r(N)$ after dualization}
			We conclude this section by recording the above analysis in the following lemma. 
			\begin{lemma} \label{aftervoronoisum}
				Let $S_r(N)$ be as in \eqref{mainsum}. Let $\widetilde{N}$ and $\widetilde{M}$ be as in \eqref{ntilde12} and  \eqref{mtilde} respectively. Then we have 
				\begin{align}  \label{aftervoronoi}
					S_{r}(N)=& \frac{N^{5/4} }{K Q r^{1/2}} \sum_{ N_0^{1-\epsilon}  \ll  \mathfrak{c} \ll N_0^{1+\epsilon}}   \sum_{1\leq q \leq Q} \frac{1}{q^2} 
					\sum_{\pm}  \sum_{n_{1}|qr}   \sum_{n_{2} \ll \tilde{N}/n_{1}^2}  \frac{A(n_{1},n_{2})}{\sqrt{n_{2}}}  \\
					& \times
					\sum_{m \ll \widetilde{M}} \frac{\lambda(m)}{m^{1/4}}    \mathcal{C}(n_{1}^2 n_{2}, q,m) \mathcal{J}(m,n_{1}^2n_{2},q) + O(T^{-2020}),\notag
				\end{align}
				where
				$ \mathcal{C}(n_{1}^2 n_{2}, q,m) $ is the character sum given by 
				\begin{align*}
					\sideset{}{^{\star}}{\sum}_{a \, \mathrm{mod} \, q} S \left(r \bar{a}, \pm n_{2}, qr/n_1\right) \, e\left(\frac{\bar{a} m}{q}\right)&=  \sideset{}{^{\star}}{\sum}_{\beta \, \mathrm{mod}\,  {qr}/{n_1}} \, e\left(\frac{ \pm \bar{\beta} n_{2}}{qr/n_{1}}\right)	\sideset{}{^{\star}}{\sum}_{a \, \mathrm{mod} \, q}e\left(\frac{\bar{a}(m+\beta n_1)}{q}\right)\\
					&= \sum_{d \vert q} d \mu\left(\frac{q}{d}\right) \sum_{\substack{\beta \;  \mathrm{mod} \;  {qr}/{n_1} \\  n_{1} \beta \;  \equiv \; - m \; \mathrm{mod} \;  d}} \; e \left(\frac{\pm \overline{\beta} n_{2}}{qr/n_{1}}\right),
				\end{align*} and 
				
				\begin{align}\label{ mathcali-1}
					\mathcal{J}(m,n_{1}^2n_{2},q)=	\int_{\mathbb{R}} W(x/Q^\epsilon) g(q,x) \int_{\mathbb{R}} V\left(\frac{v}{K}\right) G_{\star} \left({n_{2}^\star N} \right) \, I^{\prime}_{1}(m,x,q) \, \mathrm{d} v \, \mathrm{d}x,
				\end{align}
				with $I^{\prime}(m,x,q)$ and $G_{\star}(n_2^\star N)$  as in  \eqref{iprimeinte} and \eqref{gl3integral} respectively.
			\end{lemma}

			\section{An application of Cauchy and Poisson}\label{cauchy and poisson}
			\subsection{Cauchy's inequality}
			After the  Voronoi formulae, we apply  Cauchy's inequality to the sum over $n_2$  in \eqref{aftervoronoi}.  On dividing the $q$-sum and $m$-sum in \eqref{aftervoronoi} into dyadic blocks $q \sim C$, $C\ll Q$ and $m \sim M_1$, $M_1 \ll \widetilde{M}$ respectively, and   spliting  $q$ as $q=q_1q_2$ with $q_1 |(n_1r)^{\infty}$,  $(q_{2},n_1r)=1$, we see that $S_{r}(N)$ is bounded by
			\begin{align*}
				\sup_{ \substack{ C, \,  M_1 
				}}	\ \frac{T^{\epsilon} N^{5/4}}{r^{1/2} KQ C^2} \, & \sum_{\frac{n_{1}}{(n_{1},r)} \ll C}  \sum_{\frac{n_{1}}{(n_{1},r)} \vert q_{1} \vert (n_{1}r)^{\infty}}\sum_{n_2 \ll \widetilde{N}/n_1^2} \frac{| \lambda_\pi(n_{1},n_{2})| }{ n_2^{1/2}} \\ 
				& \times \Big|   \sum_{q_{2} \sim {C}/{q_{1}}} \; \sum_{m  \sim  {M_1}} \frac{\lambda_f(m) }{m^{1/4}} \;  \mathcal{C}(n_{1}^2 n_2,q,m) \, \mathcal{J}(m,n_{1}^2 n_{2},q)  \Big|,
			\end{align*} 
			Now on  applying   Cauchy's inequality to the $n_{2}$-sum we arrive at 
			\begin{align} \label{srsum}
				S_{r}(N) \ll 	\sup_{ \substack{C, \,  M_1 
				}}	\ \frac{T^{\epsilon} N^{5/4}}{r^{1/2} KQ C^2}  \sum_{\frac{n_{1}}{(n_{1},r)} \ll C} \Xi^{1/2}  \, \sum_{\frac{n_{1}}{(n_{1},r)} \vert q_{1} \vert (n_{1}r)^{\infty}} \, \Omega^{1/2}, 
			\end{align} 
			where 	$$\Xi = \sum_{n_2 \ll \widetilde{N}/n_1^2} \frac{\vert \lambda_\pi(n_{1},n_{2})\vert^2}{n_{2}},$$
			and
			\begin{equation} \label{Omega123}
				\Omega = \sum_{n_2 \ll \widetilde{N}/n_1^2} \Big \vert \sum_{q_{2} \sim {C}/{q_{1}}} \sum_{m \sim M_{1}} \frac{\lambda_f(m) }{m^{1/4}} \;  \mathcal{C}(n_{1}^2 n_2,q,m) \, \mathcal{J}(m,n_{1}^2 n_{2},q) \Big \vert^2
			\end{equation}
			
			with
			\begin{equation} \label{paravlues}
				\widetilde{N} = \frac{r Q^3 K^2 T}{N} T^{\epsilon}, \quad  M_1 \ll \widetilde{M}= \left( \frac{T^{2} C^2}{N} + K \right) T^{\epsilon}. 
			\end{equation}
			\subsection{The Poisson summation formula}
			In the next step,  we will  analyse $\Omega$ using the Poisson summation formula. We have the following lemma.
			
			\begin{lemma} \label{Omega3}
				Let $\Omega$ be as in \eqref{Omega123}. Let $q=q_1q_2$ and $q^\prime=q_1q_2^{\prime}$. Then we have 
				\begin{align} \label{omegavalue}
					\Omega \ll   \sup_{N_1 \ll \widetilde{N}} \frac{{N_1} T^{\epsilon}}{n_1^2M_{1}^{1/2} } \; & \mathop{\sum \sum}_{q_{2},\,  q_{2}^{\prime} \sim {C}/{q_{1}}}  \; \mathop{\sum \sum}_{m,\, m^{\prime} \sim M_{1}}  \sum_{n_{2} \in \mathbb{Z}} \left| \mathfrak{C}(...)\right| \; \left|\mathcal{I}(...)\right|,
				\end{align} 
				where 
				$$\mathfrak{C}(...)= \mathop{\sum \sum}_{\substack{d |q \\ d^{\prime} | q^{\prime}}} d d^{\prime} \mu \left(\frac{q}{d}\right) \mu \left(\frac{q^{\prime}}{d^{\prime}}\right) \; \mathop{\sideset{}{^\star}{\sum}_{\substack{\beta \;\mathrm{mod} \; \frac{q_{1}q_{2}r}{n_{1}} \\  n_{1} \beta \;  \equiv \; - m \; \mathrm{mod} \;  d}} \, \sideset{}{^\star}{\sum}_{\substack{\beta^{\prime} \; \mathrm{mod} \; \frac{q_{1}q_{2}^{\prime}r}{n_{1}} \\  n_{1}   \beta^{\prime} \;  \equiv \; - m^{\prime} \; \mathrm{mod} \;  d^{\prime}}}}_{\overline{\beta} q_{2}^{\prime} - \overline{\beta^{\prime}} q_{2} + n_{2} \; \equiv \; 0 \; \left(\frac{q_{1} q_{2} q_{2}^{\prime}r}{n_{1}} \right)} \; 1 $$
				and
				\begin{align}\label{integral after cauchy}
					\mathcal{I}(...)= \int V_1(y) \; \mathcal{J} \left({m, N_1} y, q\right) \; \overline{\mathcal{J}}( m^\prime, N_1 y, q^\prime ) \; e \left(\frac{-n_{2} {N_1} y}{n_{1} q_{1} q_{2} q_{2}^{\prime}r}\right) \; \mathrm{d} y,
				\end{align}
				where (recall from \eqref{ mathcali-1})  $ \mathcal{J} \left({m, N_1} y, q\right) $ is given by
				\begin{align}\label{matcalj}
					\int_{\mathbb{R}} W(x/Q^\epsilon) g(q,x) \int_{\mathbb{R}} V\left(\frac{v}{K}\right) G_{\star} \left({ \frac{N_1Ny}{q^3r}} \right) \, I^{\prime}_{1}(m,x,q) \, \mathrm{d} v \, \mathrm{d}x,
				\end{align}
				and  ${\mathcal{J}}( m^\prime, N_1 y, q^\prime ) $ is defined similarly. 
			\end{lemma}
			
			\begin{proof}
				We first split the sum over $n_2$  in \eqref{Omega123} into dyadic blocks $n_1^2n_2 \sim N_1$, $N_1 \ll \widetilde{N}$ and  introduce a non-negative smooth bump function $V_1$  supported  on $[2/3, \, 3]$ with $V_1(y)=1$ for $w \in [1,\, 2]$ and $V_1^{(j)}(w) \ll_j 1$ and  we  arrive at 
				$$\Omega \ll  T^{\epsilon} \sup_{N_1 \ll \widetilde{N}}\sum_{n_2 \in \mathbb{Z}} V\left(\frac{n_{1}^{2}n_{2}}{{N_1}}\right) \Big \vert \sum_{q_{2} \sim {C}/{q_{1}}} \sum_{m \sim M_{1}} \frac{\lambda_f(m) }{m^{1/4}} \;  \mathcal{C}(n_{1}^2 n_2,q,m) \, \mathcal{J}(m,n_{1}^2 n_{2},q) \Big \vert^2.$$
				Next we open the absolute value square and interchange the summations to arrive at 
				\begin{align*}
					T^{\epsilon} \sup_{N_1 \ll \widetilde{N}}\mathop{\sum \sum}_{q_{2}, q_{2}^{\prime} \sim {C}/{q_{1}}}  \; \mathop{\sum \sum}_{m,m^{\prime} \sim M_{1}} \frac{\lambda_{f}(m)\lambda_{f}(m^\prime)}{(mm^\prime)^{1/4}}  \sum_{n_{2} \in \mathbb{Z}}V_1\left(\frac{n_{1}^{2}n_{2}}{{N_1}}\right)\mathcal{C}(...)\overline{\mathcal{C}}(...)\mathcal{J}(...)\overline{{\mathcal{J}}}(...).
				\end{align*}
				Now on  applying the Poisson summation formula to the $n_{2}$-sum with the modulus $q_{1}q_{2}q_{2}^{\prime}r/n_{1}$ and using Delinge's bound   $|\lambda_{f}(m)| \leq m^{\epsilon}$, $|\lambda_{f}(m^\prime)| \leq m^{ \prime \epsilon}$ we get the lemma.
			\end{proof}

			\section{Estimates for the integral transform}
			In this section, we will analyse the integral transform $\mathcal{I}(...)$ defined in \eqref{integral after cauchy}.  We begin by analysing $ \mathcal{J} \left(m,{N_1} y, q\right)$, which is given by (see  \eqref{matcalj})
			\begin{align}\label{int before cauchy}
				\int_{\mathbb{R}} W(x/Q^\epsilon) g(q,x) \int_{\mathbb{R}} V\left(\frac{v}{K}\right)  G_{\star} \left({ \frac{N_1Ny}{q^3r}} \right)  \, I^{\prime}_{1}(m,x,q) \, \mathrm{d} v \, \mathrm{d}x. 
			\end{align}
			We have the following lemma.
			\begin{lemma} \label{jcalvalue}
				We have  
				\begin{align}\label{simplified j}
					\mathcal{J} \left(m,{N_1} y, q\right)=\int_{\mathbb{R}} V\left(\frac{v}{K}\right)   \int_{u \ll \frac{T^\epsilon C}{QK}}\, \mathrm{I}_u \, 	I(N_1y,m,q)
					\, \mathrm{d}u \,\mathrm{d} v + \mathrm{lower \ order \ terms},
				\end{align}
				where 
				\begin{align}
					\mathrm{I_u}=  \int_{\mathbb{R}} W(x/Q^\epsilon)  g(q,x)e \left(\frac{Nxu}{qQ}\right)\, \mathrm{d}x,
				\end{align}
				\begin{align}
					I(N_1y,m,q)= \sum_{   \mathfrak{c} } \int_{ A_{\mathfrak{c}} }  W_{3}\left( \frac{|\eta-{\bf t}_3|}{\sqrt{mN}/q}  \right)   |\eta -{\bf t}_{3}|^{-1/2} \;(N_1y)^{-i(\eta-{\bf t}_3)} e( h(\eta))  \; \mathrm{d} \eta. 
				\end{align}
				Here  $W_3$ a smooth function supported on $[1/2,5/2]$ and satisfying $x^jW_3^{(j)}(x)\ll T^{\epsilon}$ and 
				\begin{align}\label{hvalue}
					h(\eta)=-\frac{(\eta -{\bf t}_3)}{2\pi}+\frac{1}{2\pi}&\sum_{j=1}^{3} (\eta -{\bf t}_3 + {\bf t}_{j}) \, \log |\eta -{\bf t}_3 + {\bf t}_{j}| \notag \\
					&-\frac{(\eta-{\bf t}_3)}{\pi}\log\left({c|\eta-{\bf t}_{3}|}\right),
				\end{align}
				with  $$c=c(q,m):=\sqrt{\frac{8\pi^3 N}{q^3r}}\frac{q}{ 2 \pi e \sqrt{Nm} }.$$
				\begin{remark}
					Lower order terms come from the stationary phase analysis. They  have the same phase function as the main term.  Thus, on analysing them like the main term,  we save  at least $T$ extra (size of the second derivative)  at the end. 
				\end{remark}
			\end{lemma}
			\begin{proof}
				On plugging the expression of $G_{\star} \left({ \frac{N_1Ny}{q^3r}} \right)$    from  \eqref{gl3integral}  (replacing $n_2^\star$ by $N_1y/(q^3r)$)  and the expression of $ I^{\prime}_{1}(m,x,q)$ from    \eqref{iprimeinte}, we see that 
				\begin{align*}
					\mathcal{J} \left(m,{N_1} y, q\right)=&   \int_{\mathbb{R}} W(x/Q^\epsilon)  g(q,x) \int_{\mathbb{R}} W\left(\frac{v}{K}\right)  \sum_{   \mathfrak{c} } \int_{ A_{\mathfrak{c}} } W_{1}\left( \tau \right) e \left(h_{1}(\tau) \right) \\
					& \times  \int_{0}^{\infty} V \left(y_{1}\right) y_{1}^{-i({{\bf t}_3}+v)} e \left(\frac{y_{1}N x}{qQ}\right) y_{1}^{-1/2 - i \tau}  \, \mathrm{d} y_{1}\ \mathrm{d}\tau  \\
					& \times \int_{0}^{\infty} U_{\pm}(y_{2}) y_{2}^{i({\bf t}_{3}+v)} e \left(-\frac{xNy_{2}}{qQ} \pm \frac{2 \sqrt{y_{2}Nm}}{q}\right) dy_{2} \, \mathrm{d} v \, \mathrm{d}x.
				\end{align*} 
				Recall from \eqref{h_1-1} that $h_1(\tau)$ is given by 
				\begin{align}\label{h_1}
					h_1(\tau)=-\frac{\tau}{2\pi}+\frac{1}{2\pi}\sum_{j=1}^{3} (\tau + {\bf t}_j)\log|\tau +{\bf t}_j| - \frac{\tau}{2\pi} \log \left( \frac{8\pi^3  NN_1y }{q^3r}\right).
				\end{align} 
				Note that we have replaced $n_2^\star $ by  $N_1y/(q^3r)$ in \eqref{h_1-1}. On analysing  the $x$-integral or   the $v$-integral (using the intrgration by parts),  we get the restriction 
				$$y_1-y_2 \ll T^\epsilon C/QK.$$
				Thus, upon writing $y_1 = y_2 + u$ with $|u|\ll {T^\epsilon C }/{QK}$, we arrive at the following expression of    $ \mathcal{J}(m,N_1y,q)$ 
				\begin{align}\label{int after c}
					\int_{\mathbb{R}} V\left(\frac{v}{K}\right)   \int_{u \ll \frac{T^\epsilon C}{QK}}\, \mathrm{I}_u 
					\sum_{   \mathfrak{c} } \int_{ A_{\mathfrak{c}} }  W_{1}\left(\tau\right)\, \, e \left(h_{1}(\tau)\right)\, \mathrm{I_2}(u, \tau,v)\mathrm{d}\tau \, \mathrm{d}u \,\mathrm{d} v 
				\end{align}
				where  
				\begin{align}
					\mathrm{I_u}=  \int_{\mathbb{R}} W(x/Q^\epsilon)  g(q,x)e \left(\frac{Nxu}{qQ}\right)\, \mathrm{d}x,
				\end{align}
				and 
				\begin{align}
					\mathrm{I_2}(u, \tau,v)= \int_{0}^{\infty}  W_{u,v,\tau}(y_2)\,y_{2}^{ - i \tau}\, e \left( \pm \frac{2 \sqrt{y_{2}Nm}}{q}\right) \mathrm{d} y_2,
				\end{align}
				with  
				$$W_{u,v,\tau}(y_2):=\frac{V \left(y_2+u\right)U_{\pm}(y_2)}{\sqrt{y_2+u}} \left(1+\frac{u}{y_2}\right)^{-i( {\bf t}_3+v+\tau)}.$$
				Note that $\left(1+{u}/{y_2}\right)^{-i({\bf t}_{3}+v+\tau)}$ does not oscillate as a function of $y_2$ ($j$-th derivative is  bounded by $T^\epsilon$).  Thus, we see that $W^{(j)}_{u,v,\tau}(y_2) \ll_j T^\epsilon$. Now, on applying the change of variable $y_2 \rightarrow y_2^2$, we get the following expression of $	\mathrm{I_2}(u, \tau,v)$
				\begin{align}
					&\int_{0}^{\infty} 2y_2 W_{u,v,\tau}\left(y_2^2\right)\,y_{2}^{ - 2i \tau}\, e \left( \pm \frac{2 \sqrt{Nm}y_2}{q}\right) \mathrm{d} y_2, 
				\end{align}
				which we analyse now.  In  further discussion, by  abuse of notation, we will write $W_2(y_2)$ in place of   $2y_2 W_{u,v,\tau}(y_2^2)$,  as we will see that the  analysis is uniform with respect to $u$, $v$ and $\tau$. Thus the $y_2$-integral has  the following form
				\begin{align}\label{y_2 integral}
					\mathrm{I_2}(u, \tau,v)= 	\int_{0}^{\infty} W_2(y_2)\, e \left(- \frac{\tau}{  \pi} \log y_2 \pm \frac{2 \sqrt{Nm}y_2}{q}\right) \mathrm{d} y_2.
				\end{align} 
				Consider the phase function
				\begin{equation} \label{first phase function}
					f(y_2) = - \frac{\tau}{  \pi} \log y_2 \pm \frac{2 \sqrt{ N m}}{q} y_{2}. 
				\end{equation}
				The stationary point  $y_{2,0}$ is given by 
				$$ y_{2, 0} = \frac{\pm\tau q}{ 2 \pi \sqrt{Nm}}.$$
				Note that  $\pm \tau$ has  to  be positive  for the stationary point to lie in the range $[1,2]$,  otherwise the integral will be negligibly small.  Moreover 
				\begin{align}\label{tau size}
					\tau \asymp \sqrt{Nm}/q
				\end{align}
				due to the same reason.
				The second order derivative of the phase function at $y_{2,0}$ is given by  
				\begin{align} \label{K 0 definition}
					f^{\prime \prime} (y_2) = \frac{\tau}{ \pi y_2^2}  \Rightarrow f^{\prime \prime} (y_{2, 0})=  \frac{4\pi N m}{\tau q^2}.
				\end{align}
				On applying the stationary phase expansion (see  Lemma \ref{stationaryphase}), we see that the $y_2$- integral in \eqref{y_2 integral} can be expressed as 
				$$ \frac{W_2(y_{2,0})e\left(f (y_{2, 0})+1/8\right)}{\sqrt{f^{\prime \prime} (y_{2, 0})}} + \text{lower order terms}+O\left(\frac{1}{T^{2020}}\right).$$
				We will  work with the main  term, as the analysis for the other terms is  similar and we get better bounds. Thus the $y_2$-integral essentially looks like
				$$ \frac{W_2(y_{2,0})e\left(f (y_{2, 0})+1/8\right)}{\sqrt{f^{\prime \prime} (y_{2, 0})}} =\frac{W_2(y_{2,0})e\left(1/8\right)q\sqrt{|\tau|}}{\sqrt{4\pi Nm}}\left( \frac{|\tau| q}{ 2 \pi e \sqrt{Nm} }\right)^{-2i \tau}. $$
				On plugging the above expression   in place of $\mathrm{I_2}(u, \tau,v)$ into \eqref{int after c}, we see that  $ \mathcal{J}(m,N_1y,q)$  is  given by
				\begin{align}
					&	\int_{\mathbb{R}} V\left(\frac{v}{K}\right)   \int_{u \ll \frac{T^\epsilon C}{QK}}\, \mathrm{I}_u 
					\sum_{   \mathfrak{c} } \int_{ A_{\mathfrak{c}} }  W_{1}\left( \tau \right) W_2\left(\frac{|\tau| q}{ 2 \pi \sqrt{Nm}}\right)\notag \\ 
					& \times \frac{e\left(1/8\right)q\sqrt{|\tau|}}{\sqrt{4\pi Nm}}\left( \frac{|\tau| q}{ 2 \pi e \sqrt{Nm} }\right)^{-2i \tau}\, \, e \left(h_{1}(\tau)\right)\, \mathrm{d}\tau \, \mathrm{d}u \,\mathrm{d} v.
				\end{align}
				Consider 
				$$ \frac{e\left(1/8\right)q\sqrt{|\tau|}}{\sqrt{4\pi Nm}}= \frac{1}{\sqrt{|\tau|}}\frac{e\left(1/8\right)q{|\tau|}}{\sqrt{4\pi Nm}}.$$
				Note that the last term has size $1$. Thus it can be absorbed into the weight function. Finally, changing the variable $\tau + {\bf t}_{3} \mapsto \eta$, we essentially arrive at 
				\begin{align*}
					&	\int_{\mathbb{R}} V\left(\frac{v}{K}\right)   \int_{u \ll \frac{T^\epsilon C}{QK}}\, \mathrm{I}_u \, 	I(N_1y,m,q)
					\, \mathrm{d}u \,\mathrm{d} v,
				\end{align*}
				where 
				\begin{align}
					I(N_1y,m,q)= \sum_{   \mathfrak{c} } \int_{ A_{\mathfrak{c}} }  W_{3}\left( \frac{|\eta-\bf t_3|}{\sqrt{mN}/q}  \right)  |\eta -{\bf t}_{3}|^{-1/2} \;(N_1y)^{-i(\eta-{\bf t}_3)} e( h(\eta))  \; \mathrm{d} \eta. 
				\end{align}
				Here  $W_3$ is  the new weight function supported on $[1/2,5/2]$ and satisfying $y^jW_3^{(j)}(y) \ll T^{\epsilon}$  and $h(\eta)$ is given by
				\begin{align*}
					&h_{1}(\eta-{\bf t}_3) + \frac{(\eta-{\bf t}_3)}{2\pi} \log \left(N_1y\right)-\frac{(\eta-{\bf t}_3)}{\pi}\log\left(\frac{|\eta-{\bf t}_{3}| q}{ 2 \pi e \sqrt{Nm} }\right) \notag\\
					=&-\frac{(\eta -{\bf t}_3)}{2\pi}+\frac{1}{2\pi}\sum_{j=1}^{3} (\eta -{\bf t}_3 + {\bf t}_{j}) \, \log |\eta -{\bf t}_3 + {\alpha}_{j}|-\frac{(\eta-{\bf t}_3)}{\pi}\log\left({c|\eta-{\bf t}_{3}|}\right),
				\end{align*}
				with
				$$c=c(q,m):=\sqrt{\frac{8\pi^3 N}{q^3r}}\frac{q}{ 2 \pi e \sqrt{Nm} }.$$ Here we used the expression \eqref{h_1} for $h_1(\eta-\alpha_3)$. 
				Hence we have the lemma.
			\end{proof}
				Analysis for  $\mathcal{J}( m^\prime, N_1 y, q^\prime )$ is exactly similar to that of $ \mathcal{J} \left({m, N_1} y, q\right)$. Thus,  on plugging the expression \eqref{simplified j} for $ \mathcal{J} \left({m, N_1} y, q\right)$ and a corresponding expression of $\mathcal{J}( m^\prime, N_1 y, q^\prime )$ into \eqref{integral after cauchy},  we get 
				\begin{align}\label{final I}
					\mathcal{I}(...)=\int_{\mathbb{R}}\int_{\mathbb{R}} V\left(\frac{v}{K}\right)  V\left(\frac{v^\prime}{K}\right) \int_{u \ll \frac{T^\epsilon C}{QK}} \int_{u^\prime \ll \frac{T^\epsilon C}{QK}} \mathrm{I}_u  \overline{\mathrm{I}_u^\prime}\, \mathfrak{J}\,  \mathrm{d}u^\prime \, \mathrm{d}u \, \mathrm{d} v^\prime \,  \mathrm{d} v,
				\end{align}
				where 
				\begin{align} \label{J after simple}
					\mathfrak{J}=\int_{\mathbb{R}} V_1(y)  {I} \left({N_1} y, q,m\right) \; \overline{I} \left({N_1} y, q^\prime,m^\prime\right)  e \left(\frac{-n_{2} {N_1} y}{n_{1} q_{1} q_{2} q_{2}^{\prime}r}\right) \;  \; \mathrm{d}y.
				\end{align}
				Thus it boils down to analyse $\mathfrak{J}$ only, as the other integrals will be estimated trivially. 
				
				For the zero frequency, i.e., $n_2=0$, we have the following estimate. 
				\begin{lemma} \label{secondderibound}
					Let $\mathfrak{J}$ be as above. Let $N_0= T^\epsilon N/(CQ)$. 	Then  we have
					\begin{align}\label{zerointegralbound}
						\mathfrak{J} \ll  \frac{T^\epsilon N_0}{\max\{N_0,T\}} \ll  \frac{T^\epsilon N_0}{T}  .
					\end{align}
					Moreover, for $n_2=0$ and $q=q^\prime$, $\mathfrak{J}$ is negligibly small unless
					\begin{equation} \label{diffmm}
						m^{\prime}  - m  \ll {M_{1}} T^{\epsilon}/N_0,
					\end{equation}
					where $m,\, m^\prime \sim M_1$.
				\end{lemma}
				\begin{proof}
					Let's consider  $I(N_1y,m,q)$ which is given by 
					\begin{align}\label{I after cauchy}
						\sum_{   \mathfrak{c} } \int_{ A_{\mathfrak{c}} }  W_{3}\left( \frac{|\eta-{\bf t}_3|}{\sqrt{mN}/q}  \right)  |\eta -{\bf t}_{3}|^{-1/2} \;(N_1y)^{-i(\eta-{\bf t}_3)} e( h(\eta))  \; \mathrm{d} \eta,
					\end{align}
					Here $A_{\mathfrak{c}}$  is the set of all $\eta_1 \in \mathbb{R}$ such that 
					\begin{align*}
						&	\mathfrak{c} \leq |\eta_1| \leq \mathfrak{c}(1+1/10),\,   T^\epsilon \leq  |\eta_1-{\bf t}_3+{\bf t}_1| \ll (\mathfrak{c} +T)T^\epsilon, \notag  \\
						&T^\epsilon \leq  |\eta_1-{\bf t}_3+{\bf t}_2| \ll \mathfrak{c}^{1+\epsilon},
					\end{align*}
					and $\mathfrak{c} \asymp N_0$ and   the phase function  is given by  
					\begin{align*}
						h_3(\eta)&=h(\eta) -\frac{\eta-{\bf t}_3}{2\pi}\log(N_1y) \\
						&=-\frac{(\eta -{\bf t}_3)}{2\pi}+\frac{1}{2\pi}\sum_{j=1}^{3} (\eta -{\bf t}_3 + {\bf t}_{j}) \, \log |\eta -{\bf t}_3 + {\bf t}_{j}| \\
						&\hspace{1cm}-\frac{(\eta-{\bf t}_3)}{\pi}\log\left({c_1(y)|\eta-{\bf t}_{3}|}\right),
					\end{align*}
					where $$c_1(y)=\sqrt{\frac{8\pi^3{N_1}y N}{q^3r}}\frac{q}{ 2 \pi e \sqrt{Nm} }.$$
					The second order derivative of the phase function is given by 
					\begin{align}\label{second deri}
						h_3^{\prime \prime}(\eta)=-\frac{1}{\pi (\eta -{\bf t}_3)}+\frac{1}{2\pi} \sum_{j=1}^{3}\frac{1}{\eta -{\bf t}_3+{
								\bf t}_j}.
					\end{align}
					We estimate it in two cases. 
					
					\noindent
					{\bf Case 1}: $\mathfrak{c} \asymp  T $. 
					
					As $\eta \asymp \mathfrak{c}$, we observe that 
					$$ |\eta-{\bf t}_{3}| \ll T,\, |\eta-{\bf t}_{3}+{\bf t}_1| \ll T,\,  |\eta-{\bf t}_{3}+{\bf t}_2| \asymp  T, \,  \eta \asymp  T. $$
					We also recall from \eqref{tau size} that $\eta-{\bf t}_{3} \asymp \sqrt{Nm}/q$, as $\tau=\eta-{\bf t}_3$.

					Let $ |\eta-{\bf t}_{3}| \ll T^{1-\epsilon}$, then we see that $ |\eta-{\bf t}_{3}+{\bf t}_1| \asymp T$. Thus the first term dominates the rest of the terms  in \eqref{second deri}. Hence $h^{''}(\eta)\gg q/\sqrt{Nm}$. On applying the second derivative bound , we get $I(N_1y,m,q)\ll T^\epsilon$. 
					The case $ |\eta-\alpha_{3}| \asymp T$ needs to be treated differently, as 
					$h_3^{''}(\eta)$ may vanish in this situation. 
					
					\noindent
					{\bf Case 2}: $N_0 \not\asymp  T$. 
					
					Using \eqref{J after simple}, we see that
					\begin{align} 
						|\mathfrak{J}|\leq \int_{\mathbb{R}} V_1(y)\,   |{I} ({N_1} y, q,m)|^2 \; \mathrm{d}y.
					\end{align}
					Opening the absolute value square, and changing $\eta\rightarrow \mathfrak{c}\eta$,  $\eta^\prime \rightarrow \mathfrak{c}^\prime \eta^\prime$ we get
					\begin{align*}
						|\mathfrak{J} |\leq \sup_{ \mathfrak{c}, \mathfrak{c}^\prime}  \mathfrak{c} \mathfrak{c}^\prime \int \int |F(\mathfrak{c}\eta)| \, |F(\mathfrak{c}^\prime \eta ^{\prime})| \,\left| \int_{\mathbb{R}} V_1(y) \, y^{-i(\mathfrak{c}\eta - \mathfrak{c}^\prime\eta^{\prime})} \, \mathrm{d}y\right| \, \mathrm{d} \eta \, \mathrm{d} \eta^{\prime}.
					\end{align*}
					where $F(\eta)=W_{3}\left( \frac{|\eta-{\bf t}_3|}{\sqrt{mN}/q}  \right) |\eta -{\bf t}_{3}|^{-1/2}$, and $\mathfrak{c} \asymp \mathfrak{c}^\prime \asymp N_0$.  On applying  the change of variable $\eta^\prime \rightarrow \mathfrak{c}\eta^\prime/\mathfrak{c}^\prime$ followed by   integration by parts repeatedly,  we observe that the $y$-integral is negligibly small unless  $\vert \eta - \eta^{\prime}\vert \ll T^{\epsilon}/N_0.$ Now changing the variable $\eta^{\prime} = \eta + \eta_1$ with $\eta_1 \ll T^{\epsilon}/N_0$, we get   $$	|\mathfrak{J}| \leq \frac{T^\epsilon N_0}{\max\{N_0,T\}},$$  
					where we used $|N_0\eta-{\bf t}_{3}|\asymp \max\{ N_0,T\}$.

					Now we prove the second part.  Consider $\mathfrak{J}$, which is given by
					\begin{align*}
						\sum_{   \mathfrak{c}} \int_{ A_{\mathfrak{c}} }  \int_{ A_{\mathfrak{c}} }  F(\eta) \, F(\eta ^{\prime}) \, {N_1}^{-i(\eta-\eta^{\prime})} \, e{( h(\eta)- h(\eta^{\prime}))} \, \int_{\mathbb{R}} V_1(y) \, y^{-i(\eta - \eta^{\prime})} \, \mathrm{d}y \, \mathrm{d} \eta \, \mathrm{d} \eta^{\prime}.
					\end{align*}
					Note that $\mathfrak{c}=\mathfrak{c^\prime}$, as $q=q^\prime$.  On  changing $\eta\rightarrow \mathfrak{c}z$,  $\eta^\prime  \rightarrow \mathfrak{c} z^\prime$ with $1 \leq  z,\, z^\prime \leq 11/10$ and proceeding as before,   we observe that the $y$-integral is negligibly small unless 
					$$\vert z - z^{\prime}\vert \ll T^{\epsilon}/N_0.$$ 
					In this case, writing  $z^{\prime} = z + z_1$ with $z_1 \ll T^{\epsilon}/N_0$, we see that   $	\mathfrak{J}$  is  transformed into 
					$$  \sum_{   \mathfrak{c} } \mathfrak{c}^2  \int \int  V_1(y) \, \left({N_1} y \right)^{i \mathfrak{c}z_1} I_4(z_1,y) \, \mathrm{d} y \, \mathrm{d} z_1,$$
					where
					$$ I_4(z_1,y)=\int \,  F(\mathfrak{c}z) F(\mathfrak{c}(z+ z_1)) \,  e( h(\mathfrak{c}z) - h(\mathfrak{c}(z +z_1) )) \,\mathrm{d} z.$$
					%
					%
					To have  a stationary point for this integral, we must  have
					$$h^{\prime}({\mathfrak{c}z}) - h^{\prime}(\mathfrak{c}(z + z_1) ) =0.$$ 
					Evaluating the first order derivative, we get 
					\begin{align*}
						&\frac{ \mathfrak{c}}{2\pi}  \sum_{j=1}^{3} \lbrace \log {| {\bf t}_j-{\bf t}_3+ \mathfrak{c}z|} - \log {|{\bf t}_{j}-{\bf t}_{3}+ \mathfrak{c}(z + z_1) |} \rbrace \\
						& -\frac{ \mathfrak{c}}{\pi} \left\{ \log {| \mathfrak{c}z-{\bf t}_{3}|} - \log {| \mathfrak{c}z-{\bf t}_{3}+ \mathfrak{c} z_1|} \right\} - \frac{ \mathfrak{c}}{2\pi} \log\left({\frac{m^{\prime} }{m }} \right) =0
					\end{align*}
					Note that we are in the situation  where $| {\bf t}_j-{\bf t}_3+ \mathfrak{c}z| \geq T^{2\epsilon}$ for all $j$  and $| \mathfrak{c}z-{\bf t}_{3}|\geq T^{2\epsilon}$ (see \eqref{A11}). 
					On analysing it further, we arrive at 
					\begin{align*}
						- \mathfrak{c}  \sum_{j=1}^{3}  \log \left| 1+\frac{ \mathfrak{c} z_1}{ {\bf t}_{j}-{\bf t}_{3}+ \mathfrak{c} z} \right|   +2 \mathfrak{c}  \log \left| 1+\frac{ \mathfrak{c}z_1}{ \mathfrak{c}z-{\bf t}_{3}}\right|  - { \mathfrak{c}} \log\left({\frac{m^{\prime} }{m }} \right) =0
					\end{align*}
					On expanding it we arrive at
					\begin{align*}
						&	- \frac{ \mathfrak{c}^2 z_1 }{{\bf t}_{1}-{\bf t}_{3}+ \mathfrak{c}z}-  \frac{ \mathfrak{c}^2 z_1 }{{\bf t }_{2}-{\bf t}_{3}+ \mathfrak{c}z} - \frac{ \mathfrak{c}z_1}{z} \\
						&+ \frac{ 2 \mathfrak{c}^2 z_1 }{ \mathfrak{c}z - {\bf t}_{3}}  -  \mathfrak{c} \log\left({\frac{m^{\prime} }{m }} \right) + \text{lower  order  terms} =0
					\end{align*} 
					On comparing the sizes,  we get 
					$$ \mathfrak{c} \log\left({\frac{m^{\prime} }{m }} \right) \ll T^{\epsilon},$$
					from which we see that
					\begin{equation*} 
						m^{\prime}  - m  \ll \frac{M_{1}}{N_{0}} T^{\epsilon},
					\end{equation*}
					which follows  by expanding $\log (m^\prime/m)$. Hence the lemma follows. 
				\end{proof}
				For $n_2 \neq 0$, and generic $q$,  we need better estimates for $\mathfrak{J}$, which we prove  in  the following lemma.
				Recall that 
				\begin{align} 
					\mathfrak{J}=\int_{\mathbb{R}} V_1(y)  {I} \left({N_1} y, q,m\right) \; \overline{I} \left({N_1} y, q^\prime,m^\prime\right)  e \left(\frac{-n_{2} {N_1} y}{n_{1} q_{1} q_{2} q_{2}^{\prime}r}\right) \;  \; \mathrm{d}y.
				\end{align}
				\begin{lemma} \label{integralbound}
					$\mathfrak{J}$ is negligibly small unless 
					\begin{equation} \label{n2value}
						n_{2} \ll \frac{n_{1} C^{2} N_{0} r}{q_{1} {N_1}} \, T^{\epsilon} =: N_{2},
					\end{equation} 
					in  which  case,  for  $q \sim C \gg T^\epsilon QK/T$, we have 
					$$\mathfrak{J}  \ll \sqrt{\frac{n_{1} q_{1}q_{2}q_{2}^{\prime}r}{n_{2} {N_1}}} \, \frac{N_{0}}{T}.$$
				\end{lemma}
				
				\begin{proof}
					Recall that 
					\begin{align*} 
						\mathfrak{J}=\int_{\mathbb{R}} V_1(y)  {I} \left({N_1} y, q,m\right) \; \overline{I} \left({N_1} y, q^\prime,m^\prime\right)  e \left(\frac{-n_{2} {N_1} y}{n_{1} q_{1} q_{2} q_{2}^{\prime}r}\right) \;  \; \mathrm{d}y,
					\end{align*}
					where ${I} \left({N_1} y, q,m\right) $ is given by 
					\begin{align*}
						\sum_{   \mathfrak{c} } \int_{ A_{\mathfrak{c}} } W_{3}\left( \frac{|\eta-{\bf t}_3|}{\sqrt{mN}/q}  \right)  |\eta -{\bf t}_{3}|^{-1/2} \;(N_1y)^{-i(\eta-{\bf t}_3)} e( h(\eta))  \; \mathrm{d} \eta,
					\end{align*}
					and 
					\begin{align*}
						h(\eta)&=-\frac{(\eta -{\bf t}_3)}{2\pi}+\frac{1}{2\pi}\sum_{j=1}^{3} (\eta -{\bf t}_3 + {\bf t}_{j}) \, \log |\eta -{\bf t}_3 + {\bf t}_{j}| \\
						&\hspace{1cm}-\frac{(\eta-{\bf t}_3)}{\pi}\log\left({c|\eta-{\bf t}_{3}|}\right),
					\end{align*}
					with  
					\begin{align*}
						c=c(q,m):=\sqrt{\frac{8\pi^3 N}{q^3r}}\frac{q}{ 2 \pi e \sqrt{Nm} }.
					\end{align*}
					Consider the $y$-integral 
					$$ \mathrm{I}_3(\eta,\eta^\prime) :=\int_{\mathbb{R}} V_1(y) \, y^{-i(\eta - \eta^{\prime})} \, e\left(-\frac{n_{2} {N_1} y}{n_{1} q_{1} q_{2} q_{2}^{\prime}r}\right) \, \mathrm{d}y.$$
					
					By using  integration by parts $j$-times, we observe that
					$$ \mathrm{I}_3(\eta,\eta^\prime) \ll_j \left(|\eta-\eta^\prime|\frac{n_1q_1q_2q_2^\prime r}{|n_2|N_1} \right)^{j} \ll  \left(\frac{N_0n_1q_1q_2q_2^\prime r}{|n_2|N_1} \right)^{j},$$
					as $|\eta-\eta^\prime| \ll N_0$. 
					Thus  $ \mathrm{I}_3(\eta,\eta^\prime)$ and hence   $\mathfrak{J}$ is negligibly small unless 
					\begin{equation*} \label{n2value12}
						n_{2} \ll \frac{n_{1} C^{2} N_{0} r}{q_{1} {N_1}}  T^{\epsilon} =: N_{2}, 
					\end{equation*} 
					In this  range,  we analyze  $\mathfrak{J}$ using  the  stationary phase method. 
					On  plugging the expressions of   
					${I}\left({N_1} y, q,m\right)$ and   $\overline{I} \left({N_1} y, q^\prime,m^\prime\right) $ into $\mathfrak{J}$, we get 
					\begin{align}\label{i4}
						&  \sum_{   \mathfrak{c}, \mathfrak{c}^\prime} \int_{ A_{\mathfrak{c}} }  \int_{ A_{\mathfrak{c}^\prime} }   F(\eta)F(\eta^\prime) {N_1}^{-i(\eta - \eta^{\prime})}  e({\left(h(\eta)-h(\eta^{\prime})\right)}\,  \mathrm{I}_3(\eta, \eta^\prime)\, 
						\mathrm{d} \eta \, \mathrm{d} \eta^{\prime},
					\end{align}
					where $F(\eta)= W_{3}\left( \frac{|\eta-{\bf t}_3|}{\sqrt{mN}/q}  \right)  \left(\eta - {\bf t}_{3}\right)^{-1/2}$, $F(\eta^\prime)= W_{3}\left( \frac{|\eta^\prime-{\bf t}_3|}{\sqrt{mN}/q}  \right) \, \left(\eta^\prime - {\bf t}_{3}\right)^{-1/2}$ and 
					$$ \mathrm{I}_3(\eta, \eta^\prime)=	\int_{\mathbb{R}} y^{-1/2} V_2(y) \, y^{-i(\eta - \eta^{\prime})} \, e\left(-\frac{n_{2} {N_1} y}{n_{1} q_{1} q_{2} q_{2}^{\prime}r}\right) \, \mathrm{d}y,$$
					with $V_2(y)=y^{1/2} V_1(y)$.
					Let's assume  that $\eta >0$ and $\eta^\prime >0$.  The other cases can be dealt similarly.  Furthermore, we can assume that $|\eta-\eta^\prime|\gg T^\epsilon$ (otherwise there are no oscillations in $y^{-i(\eta-\eta^\prime)}$, so it can be taken  as  a weight function and consequently we save the extra generic  size, $N_0$, of $|\eta-\eta^\prime|$  while estimating  the $n_2$-length).
					The phase function  of  $ \mathrm{I}_3(\eta, \eta^\prime)$ is given by 
					$$P(y): = \, -\frac{\eta - \eta^{\prime}}{2 \pi} \,  \log y-\frac{n_{2} {N_1} y}{n_{1} q_{1} q_{2} q_{2}^{\prime}r}.$$
					The stationary point $y_{0}$  is given by  
					$$y_{0} = \frac{(\eta^{\prime}-\eta)}{2 \pi} \frac{n_{1} q_{1}q_{2}q_{2}^{\prime}r}{n_{2} {N_1}}:=(\eta^\prime -\eta) Y,$$
					where we   assume  that $\eta^{\prime} \geq \eta$, as the other case can be analysed similarly.  On estimating the second  order derivative, we get 
					$$P^{\prime \prime}(y_{0}) = \frac{(\eta - \eta^{\prime})}{2 \pi y_{0}^{2}} = -\frac{n_{2} {N_1}}{n_{1} q_{1}q_{2}q_{2}^{\prime}r} \, \frac{1}{y_{0}}.$$
					Hence, by the stationary phase method, i.e.,  Lemma \ref{stationaryphase} ( with $P^{(j)}(y) \asymp |\eta-\eta^\prime |\gg T^\epsilon$ for $j\geq 2$),   $ \mathrm{I}_3(\eta, \eta^\prime)$ is given by 
					\begin{align*}
						\frac{e\left(P(y_{0})\right)}{\sqrt{\vert P^{\prime \prime}(y_{0})\vert}} \Big\{\frac{V_2(y_{0})}{\sqrt{y_{0}}} + \text{lower order terms} \Big\}.
					\end{align*}
					We  work with the leading term, and thus $\mathrm{I}_3(\eta, \eta^\prime)$ is essentially given by
					\begin{align}\label{I3}
						& \sqrt{\frac{n_{1} q_{1}q_{2}q_{2}^{\prime}r}{n_{2} {N_1}}} \, \, V_2(y_{0}) \, e\left(P(y_{0})\right)  \\
						& =\sqrt{2 \pi Y} \, \, V_2(Y(\eta^{\prime} - \eta)) \, e^{i\left(\eta^{\prime} -\eta \right) \log \left(Y(\eta^{\prime}-\eta)\right) + i (\eta - \eta^{\prime}) }.\notag
					\end{align}
					Now we estimate the $\eta$ and $\eta^\prime$-integrals. On plugging the expression \eqref{I3} of $\mathrm{I}_3(\eta, \eta^\prime)$ into \eqref{i4}, we arrive at  the following expression of $\mathfrak{J}$
					\begin{align*}
						& \sqrt{2 \pi Y} \, \sum_{   \mathfrak{c}, \mathfrak{c}^\prime} \int_{ A_{\mathfrak{c}} }  \int_{ A_{\mathfrak{c}^\prime} } \, V_2(Y(\eta^{\prime} - \eta)) F(\eta)F(\eta^\prime) \, e^{iH_1(\eta,\eta^\prime)}  \mathrm{d} \eta \, \mathrm{d} \eta^{\prime},
					\end{align*}
					where $$ H_1(\eta,\eta^\prime)= 2\pi (h(\eta)-h(\eta^{\prime}))+ (\eta^{\prime}- \eta) \log {{N_1}}+ \left(\eta^{\prime} -\eta \right) \log \left(Y(\eta^{\prime}-\eta)\right) +  (\eta - \eta^{\prime}).$$
					Using the Fourier inversion, we write 
					$$V_2(Y(\eta^{\prime}- \eta)) = \int_{\mathbb{R} } \widehat{V_2}(s) \, e\left(Y(\eta^{\prime} - \eta) s \right) \, \mathrm{d}s.$$
					Note that $\widehat{V_2}(s) \ll (1+|s|)^{-2021}$. Thus we see that $\mathfrak{J}$ is given by
					\begin{align}\label{eta}
						\sqrt{2\pi Y} \, \int_{-T^{\epsilon/2}}^{T^{\epsilon/2}} \, \widehat{V_2}(s)  \sum_{   \mathfrak{c}, \mathfrak{c}^\prime} \int_{ A_{\mathfrak{c}} }  \int_{ A_{\mathfrak{c}^\prime} } \,F(\eta)F(\eta^\prime) \, e^{i H(\eta, \, \eta^{\prime},\, s)}\, \mathrm{d} \eta \, \mathrm{d} \eta^{\prime} \, \mathrm{d}s,
					\end{align}
					where $H\left(\eta, \eta^{\prime}, s\right)$ is given by
					\begin{align*}
						& \sum_{j=1}^{3} \; \left({\bf t}_{j}-{\bf t}_{3}+\eta\right) \; \log |{\bf t}_{j}-{\bf t}_{3}+\eta | -2{(\eta-{\bf t}_3)}\log\left({c|\eta-{\bf t}_{3}|}\right) \\
						&  -\sum_{j=1}^{3} \; \left({\bf t}_{j}-{\bf t}_{3}+\eta^\prime \right) \; \log |{\bf t}_{j}-{\bf t}_{3}+\eta^\prime|+2{(\eta^\prime-{\bf t}_3)}\log\left({c^\prime|\eta^\prime-{\bf t}_{3}|}\right) \\
						&  + {2\pi s Y (\eta^{\prime}- \eta) } + \left(\eta^{\prime} -\eta \right) \log \left({N_1}Y(\eta^{\prime}-\eta)\right) .
					\end{align*}
					On 	computing the  partial derivatives, we see that  $ \frac{\partial}{\partial \eta} H(\eta, \eta^{\prime},s)$ is given by
					$$\sum_{j=1}^3\log |{\alpha}_{j}-{\bf t}_{3}+\eta |-2\log(c|\eta-{\bf t}_{3}|) -2\pi sY-\log(N_1Y(\eta^\prime-\eta)),$$
					and  $ \frac{\partial}{\partial \eta^\prime} H(\eta, \eta^{\prime},s)$ is given by
					$$-\sum_{j=1}^3\log |{\bf t}_{j}-{\bf t}_{3}+\eta^\prime |+2\log(c^\prime|\eta^\prime-{\bf t}_{3}|) +2\pi sY+\log(N_1Y(\eta^\prime-\eta)).$$
					The second order partial derivatives are given by 
					\begin{align}\label{eta2}
						\frac{\partial^2}{\partial \eta^2} H(\eta, \eta^{\prime},s) = -  \frac{2}{\eta - {\bf t}_{3}} + \frac{1}{{\bf t}_{1}-{\bf t}_{3}+ \eta} + \frac{1}{{\bf t}_{2}-{\bf t}_{3}+ \eta} + \frac{1}{\eta}+ \frac{1}{\eta^{\prime}-\eta},
					\end{align}
					\begin{align}\label{eta2p}
						\frac{\partial^2}{\partial {\eta^{\prime}}^2}H(\eta, \eta^{\prime},s) =   \frac{2}{\eta^{\prime} - {\bf t}_{3}} - \frac{1}{{\bf t}_{1}-{\bf t}_{3}+ \eta^{\prime}} - \frac{1}{{\bf t}_{2}-{\bf t}_{3}+ \eta^{\prime}} - \frac{1}{\eta^{\prime}} -\frac{1}{\eta^{\prime}-\eta},
					\end{align}
					and 
					\begin{align}\label{etaetap}
						\frac{\partial^2}{\partial \eta \partial \eta^\prime} H(\eta, \eta^{\prime},s)=\frac{\partial^2}{\partial \eta^{\prime} \partial \eta} H(\eta, \eta^{\prime},s) = - \frac{1}{\eta^{\prime} - \eta}.
					\end{align}
					Since $C \gg T^\epsilon QK/T$, we have  $\mathfrak{c} \asymp \mathfrak{c}^\prime \asymp N_{0} \ll T^{1-\epsilon}$.  Let's first assume that 
					$\mathfrak{c} \asymp \mathfrak{c}^\prime \gg K^{1+\epsilon}$. Thus it follows that 
					$$ \eta-{\bf t}_3 \asymp  \eta^\prime-{\bf t}_3\asymp  \eta-{\bf t}_3+{\bf t}_1 \asymp  \eta^\prime-{\bf t}_3+{\bf t}_1  \asymp T.$$
					Let $L\leq \eta^\prime-\eta <2L$, where $T^\epsilon\leq L\ll N_0$. Note that 
					$$ \frac{1}{{\bf t}_{2}-{\bf t}_{3}+ \eta} + \frac{1}{\eta}+ \frac{1}{\eta^{\prime}-\eta} = \frac{1}{{\bf t}_{2}-{\bf t}_{3}+ \eta} + \frac{2}{\eta}+ \frac{1}{\eta^{\prime}-\eta}+O\left( \frac {N_0^{-\epsilon}}{\eta}\right) \asymp \frac{1}{L}.$$
					Thus  $	\frac{\partial^2}{\partial \eta^2} H(\eta, \eta^{\prime},s) \asymp \frac{1}{L}.$ Similarly,  $	\frac{\partial^2}{\partial \eta^\prime2} H(\eta, \eta^{\prime},s) \asymp \frac{1}{L}$. We also note that 
					\begin{align*}
						\frac{\partial^2}{\partial \eta^2} H(\eta, \eta^{\prime},s) \,\frac{\partial^2}{\partial {\eta^{\prime}}^2} H(\eta, \eta^{\prime},s) - \left(\frac{\partial^2}{\partial \eta^{\prime} \partial \eta} H(\eta, \eta^{\prime},s)\right)^{2} \gg \frac{1}{L^2}.
					\end{align*}
					Thus we apply the two dimensional exponential bound,  Corollary \ref{twodimexpointe}, to  get 
					$$ \mathfrak{J} \ll \sqrt{\frac{n_{1} q_{1}q_{2}q_{2}^{\prime}r}{n_{2} {N_1}}} \, \frac{N_{0}}{T},$$ 
					where the factor $T$ comes from the variation, as $\eta-{\bf t}_3 \asymp T$.  The other case when $N_0 \asymp K$ can be dealt with similarly. Hence we have the lemma.
					%
					
					%
				\end{proof}
				We conclude this section by giving  final estimates for $\mathcal{I}(...)$. 
				\begin{lemma}\label{final int}
					let $\mathcal{I}(...)$ be the integral transform given in \ref{final I}. Then we have 
					\begin{align}\label{final zero}
						\mathcal{I}(...) \ll T^\epsilon\frac{C^2}{Q^2} \frac{N}{CQT}.
					\end{align}
					Moreover, for $C \gg T^\epsilon QK/T$ and $n_2\neq 0$, we have 
					\begin{align}\label{final nonzero}
						\mathcal{I}(...) \ll T^\epsilon\frac{C^2}{Q^2}   \sqrt{\frac{n_{1} q_{1}q_{2}q_{2}^{\prime}r}{n_{2} {N_1}}}\frac{N}{CQT}.
					\end{align}
				\end{lemma}
				\begin{proof}
					Recall that 
					\begin{align*}
						\mathcal{I}(...)=\int_{\mathbb{R}}\int_{\mathbb{R}} V\left(\frac{v}{K}\right)  V\left(\frac{v^\prime}{K}\right) \int_{u \ll \frac{T^\epsilon C}{QK}} \int_{u^\prime \ll \frac{T^\epsilon C}{QK}} \mathrm{I}_u  \overline{\mathrm{I}_u^\prime}\, \mathfrak{J}\,  \mathrm{d}u^\prime \, \mathrm{d}u \, \mathrm{d} v^\prime \,  \mathrm{d} v.
					\end{align*}
					We estimate the  $u$-integral as 
					\begin{align*}
						\int_{u \ll \frac{T^\epsilon C}{QK }}\,|\mathrm{I}_u|\,  \mathrm{d}u  \ll \int_{u \ll \frac{T^\epsilon C}{QK}}  \,\int_{\mathbb{R}}W(x/Q^{\epsilon})|g(q,x)|\mathrm{d}x\,   \mathrm{d}u \ll  \frac{T^\epsilon C}{Q K}Q^\epsilon ,
					\end{align*}
					where we used properties  (see \eqref{g properties}) of $g(q,x)$. The same bound holds for the $u^\prime$-integral as well. Finally, on using the above bound,  bounds for $\mathfrak{J}$ from Lemma  \ref{secondderibound} and Lemma \ref{integralbound}, and estimating the $v$ and $v^\prime$-integral trivially we get the lemma. 
				\end{proof}

				\section{Estimates for the zero frequency: $n_{2}=0$}
				In this section  we will  estimate $\Omega$  given in \eqref{omegavalue} further and  use it to get bounds for $S_{r}(N)$ for $n_{2}=0$. Let $\Omega_{0}$ denotes the contribution of the zero frequency $n_{2}=0$ to $\Omega$ in \eqref{omegavalue}. Then we have the following lemma.

				\begin{lemma} \label{omega0bound}
					For $n_2=0$, we have
					$$\Omega_{0}  \ll  \frac{\widetilde{N} N_{0} M_{1}^{1/2} C^5 r}{Q^2n_{1}^3 q_{1} T} T^{\epsilon},$$
				\end{lemma}
				
				\begin{proof}
					Recall that 
					\begin{align*} 
						\Omega_0 \ll   \sup_{N_1 \ll \widetilde{N}} \frac{{N_1} T^{\epsilon}}{n_1^2M_{1}^{1/2} } \; & \mathop{\sum \sum}_{q_{2},\,  q_{2}^{\prime} \sim {C}/{q_{1}}}  \; \mathop{\sum \sum}_{m,\, m^{\prime} \sim M_{1}}   \left| \mathfrak{C}(...)\right| \; \left|\mathcal{I}(...)\right|,
					\end{align*} 
					where 
					$$\mathfrak{C}(...)= \mathop{\sum \sum}_{\substack{d |q \\ d^{\prime} | q^{\prime}}} d d^{\prime} \mu \left(\frac{q}{d}\right) \mu \left(\frac{q^{\prime}}{d^{\prime}}\right) \; \mathop{\sideset{}{^\star}{\sum}_{\substack{\beta \;\mathrm{mod} \; \frac{q_{1}q_{2}r}{n_{1}} \\  n_{1} \beta \;  \equiv \; - m \; \mathrm{mod} \;  d}} \, \sideset{}{^\star}{\sum}_{\substack{\beta^{\prime} \; \mathrm{mod} \; \frac{q_{1}q_{2}^{\prime}r}{n_{1}} \\  n_{1}   \beta^{\prime} \;  \equiv \; - m^{\prime} \; \mathrm{mod} \;  d^{\prime}}}}_{\overline{\beta} q_{2}^{\prime} - \overline{\beta^{\prime}} q_{2} + n_{2} \; \equiv \; 0 \; \left(\frac{q_{1} q_{2} q_{2}^{\prime}r}{n_{1}} \right)} \; 1.$$
					Note that the congruence equation 
					$$\overline{\beta} q_{2}^{\prime} - \overline{\beta^{\prime}} q_{2} +n_2 \; \equiv \; 0 \; \left(\frac{q_{1} q_{2} q_{2}^{\prime}r}{n_{1}} \right)$$
					gives  $q_{2}=q_{2}^{\prime}$ and $\beta = \beta^{\prime}$ as $n_{2}=0$.
					Therefore the character sum is bounded by
					\begin{align} \label{zerochar}
						\mathfrak{C}(...) \leq 	\mathop{\sum \sum}_{d,d^\prime \vert q} d d^\prime \sideset{}{^\star}{\sum}_{\substack{\beta \; \mathrm{mod} \; \frac{q_{1}q_{2}r}{n_{1}} \\  n_{1} \beta \; \equiv \; -m \; (d) \\  n_{1} \beta \; \equiv \; -m^\prime (d^\prime)}} 1 & \leq  \mathop{\sum \sum}_{\substack{d,d^\prime |q \\ (d,d^\prime) \vert m-m^\prime}} d d^\prime \frac{ qr}{ \,n_1 [d,d^\prime]}(n_1,m) \notag \\
						& = \frac{qr}{n_1} \mathop{\sum \sum}_{\substack{d,d^\prime |q \\ (d,d^\prime) \vert m-m^\prime}} \; (d,d^\prime)(n_1,m),
					\end{align}
					where $[d,d^\prime] $ and $(d,d^\prime)$ denote the least common multiple and greatest common divisor of $d$ and $d^\prime$ respectively. Now on substituting  the above  bound for the character sum $\mathfrak{C}(...)$,  the bound  \eqref{final zero} for the integral $\mathcal{I}(...)$ from Lemma \ref{final int},  and the restriction \eqref{diffmm} on the difference of $m$  and $m^{\prime}$,  into  $\Omega_0$ we  obtain 
					$$\Omega_{0} \leq \,  \frac{ T^\epsilon \widetilde{N} N_{0} C^2}{Q^2M_{1}^{1/2} n_{1}^3 T} \;  \mathop{\sum }_{q_{2} \sim {C}/{q_{1}}} {q r} \, \mathop{\sum \sum}_{d,d^{\prime} \mid q} \, (d,d^{\prime}) \; \mathop{ \mathop{\sum \sum}_{m,m^{\prime} \sim M_{1}}}_{\substack{ m^{\prime}  - m  \ll {T^\epsilon M_{1}}/{N_{0}}   \\ (d,d^{\prime}) \mid m-m^{\prime}}} \, (n_1,m).$$
					We observe that 
					\begin{align*}
						\mathop{\mathop{\sum \sum}_{m,m^{\prime} \sim M_{1}}}_{\substack{m^{\prime}  - m   \ll \frac{M_{1}}{N_{0}} T^{\epsilon} \\ (d,d^{\prime}) \mid m-m^{\prime}}} \, (n_1,m) \, & \ll \, M_{1} + \frac{M_{1}^2}{N_{0} \, (d,d^{\prime})} T^{\epsilon}.
					\end{align*} 
					Plugging this into the above expression of $\Omega_{0}$, we arrive at
					\begin{align*}
						\Omega_{0} & \ll  \frac{\widetilde{N}  r N_{0} C^2}{Q^2M_{1}^{1/2} n_{1}^3 T}  \sum_{q_{2} \sim {C}/{q_{1}}} q \mathop{\sum \sum}_{d,d^{\prime} \mid q} (d,d^{\prime})    \left(M_{1} + \frac{M_{1}^2}{N_{0} \, (d,d^{\prime})} \right) T^{\epsilon} \\
						& \ll  \frac{\widetilde{N}  r  C^2}{Q^2 M_{1}^{1/2} n_{1}^3 T} \left(\frac{N_0 M_{1}C^{3}}{q_{1}}+\frac{M_{1}^2 C^2}{ q_{1}}\right) T^{\epsilon}\\
						& \ll \frac{\widetilde{N}  r  C^2}{Q^2  n_{1}^3 T} \left( \frac{N_0 M_{0}^{1/2}C^{3}}{q_{1}}+\frac{M_{0}^{3/2} C^2}{ q_{1}}\right) T^{\epsilon}\\
						&\ll  \frac{\widetilde{N} N_{0} M_{1}^{1/2} C^5 r}{Q^2n_{1}^3 q_{1} T} T^{\epsilon}.
					\end{align*}
					In the last step we used  $$ \frac{M_{0}^{3/2} C^2}{N_{0} q_{1}} \ll \frac{M_{0}^{1/2}C^{3}}{q_{1}} \iff {M_1Q} \ll N \iff \frac{T^2Q^2}{N}Q \ll N \iff \frac{T^4}{K^3} \ll N,$$
					The last inequality is justified  with our choice of $K$. Hence
					the lemma follows.
				\end{proof}
				
				\subsection{Estimates for $S_{r}(N)$ in the zero frequency case}
				Let $S_{r,0}(N)$   denotes  the contribution of $\Omega_{0}$ to   $S_{r}(N)$.  Then we have the following lemma.
				\begin{lemma}\label{zerofrequencybound}
					Let $S_{r,0}(N)$ be defined as above. Then we have
					\begin{equation} 
						S_{r,0}(N) \ll  r^{1/2} N^{3/4} T^{1/2} K^{1/4} T^{\epsilon}.
					\end{equation}
				\end{lemma}
				\begin{proof}
					We recall from \eqref{srsum} that 
					\begin{align*} 
						S_{r,0}(N) \ll 	\sup_{ \substack{C, \,  M_1 
						}}	\ \frac{T^{\epsilon} N^{5/4}}{r^{1/2} KQ C^2}  \sum_{\frac{n_{1}}{(n_{1},r)} \ll C} \Xi^{1/2}  \, \sum_{\frac{n_{1}}{(n_{1},r)} \vert q_{1} \vert (n_{1}r)^{\infty}} \, \Omega_0^{1/2}.
					\end{align*} 
					On substituting the bound for $\Omega_{0}$ from  Lemma \ref{omega0bound} in the above expression, we see that 
					\begin{align*}
						S_{r,0}(N) & \ll \sup_{ \substack{C, \,  M_1 
						}}	\ \frac{T^{\epsilon} N^{5/4}}{r^{1/2} KQ C^2}  \sum_{\frac{n_{1}}{(n_{1},r)} \ll C} \Xi^{1/2}  \, \sum_{\frac{n_{1}}{(n_{1},r)} \vert q_{1} \vert (n_{1}r)^{\infty}} \, \left(\frac{\widetilde{N} N_{0} M_{1}^{1/2} C^5 r}{ Q^2n_{1}^3 q_{1} T} \right)^{1/2}   \\
						& \ll  \sup_{ \substack{C, \,  M_1 
						}}	\ \frac{T^{\epsilon} N^{5/4}}{r^{1/2} KQ C^2} \, \left(\frac{\widetilde{N} N_{0} M_{1}^{1/2} C^5 r}{Q^2 T} \right)^{1/2} \, \sum_{n_{1} \ll Cr} \frac{\Xi^{1/2}}{n_1^{2}}(n_1,r)^{1/2}.
					\end{align*}
					Using Cauchy's inequality and the partial summation formula, we observe that 
					\begin{align} \label{Xivalue}
						\sum_{n_{1} \ll Cr} \frac{\Xi^{1/2}}{n_1^{2}}(n_1,r)^{1/2} \ll \left[\sum_{n_{1} \ll Cr}\frac{(n_1,r)}{n_1^2}\right]^{1/2}\left[\mathop{\sum \sum}_{n_1^2n_2 \ll \tilde{N}} \frac{\vert \lambda_\pi(n_{1},n_{2})\vert^2}{n_1^2n_{2}}\right]^{1/2} \ll T^{\epsilon}.
					\end{align}
					Now on  plugging in the values  $N_0=N/(CQ)$, $\widetilde{N} = \frac{r Q^3 K^2 T}{N} T^{\epsilon}$, $ \widetilde{M}= \left( \frac{T^{2} C^2}{N} + K \right) T^{\epsilon}$  and the above estimate we arrive at  
					\begin{align*} 
						S_{r,0}(N) &\ll\sup_{ \substack{C \ll Q
						}}	 \frac{T^{\epsilon} N^{5/4}}{r^{1/2} KQ^3 C^2} \frac{r^{1/2} Q^{3/2}K }{N^{1/2}} \left(\frac{N}{CQ}\right)^{1/2}\left(\frac{T^{1/2} C^{1/2}}{N^{1/4}}+K^{1/4}\right) C^{5/2}  \\
						& \ll   \frac{T^{\epsilon} N^{5/4}}{r^{1/2} KQ^3 } \frac{r^{1/2} Q^{3/2}K }{N^{1/2}} \left(\frac{N}{Q}\right)^{1/2}  \left(\frac{T^{1/2} Q^{1/2}}{N^{1/4}}+K^{1/4}\right) \\
						& \ll  \frac{T^{\epsilon} N^{5/4}}{r^{1/2} KQ^3 } \frac{r^{1/2} Q^{3/2}K }{N^{1/2}} \left(\frac{N}{Q}\right)^{1/2} \frac{T^{1/2} Q^{1/2}}{N^{1/4}} \\
						& \ll r^{1/2} N^{3/4} T^{1/2} K^{1/4} T^{\epsilon}. \notag
					\end{align*}
					Hence the lemma follows.
				\end{proof}

				\section{Non-zero frequencies}
				Let $\Omega_{\neq 0}$ denote the contribution of the  non-zero frequencies $n_{2} \neq 0$ to $\Omega$ given in \eqref{omegavalue} and $S_{r, \neq 0}(N)$ denote the contribution of $\Omega_{\neq 0}$ to the sum  $S_{r}(N)$  given in \ref{srsum}. In this section, we  estimate  $\Omega_{\neq 0}$ and using this bound we  will  deduce final  estimates  for $S_{r,\neq 0}(N)$.
				We have the following lemma.
				\begin{lemma} \label{non zero omega bound}
					Let $\Omega_{\neq 0}$ be as above. Let $\Omega_{\neq 0, \, \mathrm{generic}}$ and $\Omega_{\neq 0, \, \mathrm{small}}$ denote the contribution of $C \gg  T^{\epsilon}QK/T$ and $C \ll T^\epsilon QK/T$ respectively to $\Omega_{\neq 0}$.  Then we have
					\begin{align}\label{omeganon}
						\Omega_{\neq 0,\, \mathrm{generic}} \ll \frac{N^{3/2}}{Q^{7/2}T}\frac{ r^2 } {n_{1}^2 q_1}\frac{C^{9/2}}{M_{1}^{1/2}} \left(\frac{C^2n_1}{q_1^2}+\frac{Cn_1M_1}{q_1}+{M_{1}^2}\right),
					\end{align}
					and 
					\begin{align}\label{omega small}
						\Omega_{\neq 0,\, \mathrm{small}} \ll \frac{\sqrt{N}}{\sqrt{CQ}}\frac{N^{3/2}}{Q^{7/2}T}\frac{ r^2 } {n_{1}^2 q_1}\frac{C^{9/2}}{M_{1}^{1/2}} \left(\frac{C^2n_1}{q_1^2}+\frac{Cn_1M_1}{q_1}+{M_{1}^2}\right).
					\end{align}
				\end{lemma}
				\begin{proof}
					We begin the proof by recalling $\Omega_{\neq 0}$ from \eqref{omegavalue}. It is given as 
					\begin{align*} 
						\Omega_{\neq 0} \ll  \sup_{N_1 \ll \widetilde{N}} \frac{{N_1} T^{\epsilon}}{n_1^2M_{1}^{1/2} }\; & \mathop{\sum \sum}_{q_{2}, q_{2}^{\prime} \sim {C}/{q_{1}}}  \; \mathop{\sum \sum}_{m,m^{\prime} \sim M_{1}}  \sum_{n_{2} \in \mathbb{Z}-\left\lbrace 0\right\rbrace} \left| \mathfrak{C}(...)\right| \; \left|\mathcal{I}(...)\right|.
					\end{align*} 
					The analysis for $ \mathfrak{C}(...)$ is exactly same as that of the character sum in Munshi  [\cite{munshi12}, Lemma 3]. Thus using the bound 
					$$\mathfrak{C}(...) \ll \frac{q_{1}^2 \, r (m,n_{1})}{n_{1}} \mathop{\sum \sum}_{\substack{d_{2} \mid (q_{2}, \,  n_{1} q_{2}^{\prime}- mn_{2}) \\ d_{2}^{\prime} \mid (q_{2}^{\prime},\,   n_{1} q_{2} + m^{\prime} n_{2})}} \, d_{2} d_{2}^{\prime}, $$
					we see that  
					\begin{align}\label{omega for small m}
						\Omega_{\neq 0} &  \ll  \sup_{N_1 \ll \widetilde{N}}\frac{q_{1}^2 {N_1} r T^{\epsilon}} {n_{1}^3 M_{1}^{1/2}} \mathop{\sum \sum}_{d_{2}, d_{2}^{\prime} \leq C/q_{1} } d_{2} d_{2}^{\prime} \, \mathop{\sum \sum }_{\substack{q_{2},q_2^\prime \sim   {C}/{q_{1}} \\ d_2 |q_{2}, d_2^\prime |q_{2}^{\prime}}} \mathop{ \mathop{\sum \  \sum  \ \ \sum}_{\substack{m,m^{\prime} \sim M_{1}  n_2 \in \mathbb{Z}-\{0\} \\n_{1} q_{2}^{\prime} - m n_{2} \, \equiv \,  0 \, \mathrm{mod} \, d_{2} \\  n_{1} q_{2}+ m^{\prime} n_{2} \, \equiv \,  0 \, \mathrm{mod} \, d_{2}^{\prime} }}} \, (m,n_{1}) \vert \mathcal{I}(...)\vert \notag  \\
						& = \sup_{N_1 \ll \widetilde{N}}\frac{q_{1}^2 {N_1} r T^{\epsilon}} {n_{1}^3 M_{1}^{1/2}} \mathop{\sum \sum}_{d_{2}, d_{2}^{\prime} \leq C/q_{1} } d_{2} d_{2}^{\prime} \, \mathop{\sum \sum }_{\substack{q_{2} \sim {C}/{d_{2}q_{1}} \\ q_{2}^{\prime} \sim {C}/{d_{2}^{\prime} q_{1}}}} \mathop{ \mathop{\sum \  \sum  \ \ \sum}_{\substack{m,m^{\prime} \sim M_{1}  n_2 \in \mathbb{Z}-\{0\} \\n_{1} q_{2}^{\prime}d_2^\prime - m n_{2} \, \equiv \,  0 \, \mathrm{mod} \, d_{2} \\  n_{1} q_{2}d_2+ m^{\prime} n_{2} \, \equiv \,  0 \, \mathrm{mod} \, d_{2}^{\prime} }}} \, (m,n_1)\vert \mathcal{I}(...)\vert.  
					\end{align}
					We now count the number of $(m,m^\prime)$. Since $(n_{1},d_{2})=1$, we note that
					\begin{align}
						\sum_{\substack{m \sim M_{1} \\ n_{1} q_{2}^{\prime} d_{2}^{\prime}- m n_{2} \, \equiv \,  0 \, \rm mod \, d_{2}}} (m,n_{1}) & \leq \sum_{\ell \mid n_{1}} \ell \,  \sum_{\substack{m \sim M_{1}/\ell \\ n_{1} q_{2}^{\prime} d_{2}^{\prime} \bar{\ell}- m n_{2} \, \equiv \,  0 \, \rm mod \, d_{2}}} 1 
						& \ll n_1^{\epsilon}(d_{2},n_{2}) \, \left(n_{1}+\frac{M_{1}}{d_{2}}\right).  \label{mcount}
					\end{align}
					Counting the number of $m^\prime$ in a similar fashion we get that the number of $(m,m^\prime)$ pairs is dominated by
					$$ (d_{2}^{\prime}, n_{1} q_{2} d_{2}) \, (d_{2}, n_{2}) \left(n_{1}+\frac{M_{1}}{d_{2}}\right) \left(1+\frac{M_{1}}{d_{2}^{\prime}}\right).$$
					Using this bound in the above expression of $\Omega_{\neq 0}$ and executing the sum over $q_2^\prime$, we arrive at
					\begin{align*}
						\Omega_{\neq 0}\ll & T^\epsilon \sup_{N_1 \ll \widetilde{N}}\frac{q_{1} {N_1}C r } {n_{1}^3 M_{1}^{1/2}} \mathop{\sum \sum}_{d_{2}, d_{2}^{\prime} \leq \frac{C}{q_1}} d_{2}  \mathop{ \sum }_{\substack{q_{2} \sim \frac{C}{d_{2}q_{1}} }}\mathop{\sum_{ \substack{n_{2} \ll N_2 \\ n_2 \neq 0}}} \left|\mathcal{I}(...)\right| \\
						& \times (d_{2}^{\prime}, n_{1} q_{2} d_{2}) \, (d_{2}, n_{2}) \left(n_{1}+\frac{M_{1}}{d_{2}}\right) \left(1+\frac{M_{1}}{d_{2}^{\prime}}\right).
					\end{align*}
					Note that we have restricted the range of $n_2$ upto $N_2$ due to Lemma \ref{integralbound}. 
					Let's consider the sum over $n_2$. We evaluate it in two cases. \\
					{\bf Case 1.}  $C \gg T^{\epsilon}QK/T$. \\
					In this case, using the bound
					\begin{align*}
						\mathcal{I}(...) \ll T^\epsilon\frac{C^2}{Q^2}   \sqrt{\frac{n_{1} C^2 r}{n_{2} {q_1N_1}}}\frac{N_0}{T},
					\end{align*}
					
					\begin{equation} 
						n_{2} \ll \frac{n_{1} C^{2} N_{0} r}{q_{1} {N_1}} \, T^{\epsilon} =: N_{2},
					\end{equation} 
					from Lemma \ref{final int}, we see that  
					\begin{align*}
						\mathop{\sum_{ \substack{n_{2} \ll N_2 \\ n_2 \neq 0}}}(d_2,n_2) \left|\mathcal{I}(...)\right|  & \ll \frac{C^2T^\epsilon}{Q^2}\sqrt{\frac{n_{1}C^2r}{q_1 {N_1}}} \frac{N_{0}}{T}\mathop{\sum_{ \substack{n_{2} \ll N_2 \\ n_2 \neq 0}}}\frac{(d_2,n_2)}{\sqrt{n_2}}  \\
						& \ll  \frac{C^2T^\epsilon}{Q^2} \sqrt{\frac{n_{1}C^2r}{q_1 {N_1}}} \frac{N_{0}}{T}\sqrt{N_2}=  \frac{C^2T^\epsilon}{Q^2} \frac{N_2}{\sqrt{N_0}}\frac{N_0}{T},
					\end{align*}
					where we used $N_2= \frac{n_{1} C^{2} N_{0} r}{q_{1} {N_1}} \, T^{\epsilon}$.
					On substituting  this bound and executing  the sum over $d_{2}^{\prime}$, we arrive at
					\begin{align*}
						\sup_{N_1 \ll \widetilde{N}}\frac{q_{1} {N_1}C r } {n_{1}^3 M_{1}^{1/2}} \frac{N_2}{\sqrt{N_0}}\frac{N_0C^2}{TQ^2} \mathop{ \sum}_{d_{2} \leq \frac{C}{q_1}} d_{2}  \mathop{ \sum }_{\substack{q_{2} \sim \frac{C}{d_{2}q_{1}} }} \left(n_{1}+\frac{M_{1}}{d_{2}}\right) \left(\frac{C}{q_1}+{M_{1}}\right).
					\end{align*}
					Now executing the remaining sums, we see that 
					\begin{align*}
						\Omega_{\neq 0,\, \mathrm{generic}} &\ll   \sup_{N_1 \ll \widetilde{N}}\frac{q_{1} {N_1}C r } {n_{1}^3 M_{1}^{1/2}}  \frac{N_2}{\sqrt{N_0}}\frac{N_0C^2}{TQ^2}\frac{C}{q_1} \left(\frac{Cn_1}{q_1}+{M_{1}}\right) \left(\frac{C}{q_1}+{M_{1}}\right) \\
						&\ll   \sup_{N_1 \ll \widetilde{N}}\frac{q_{1} {N_1}C r } {n_{1}^3 M_{1}^{1/2}} \frac{N_2}{\sqrt{N_0}}\frac{N_0 C^2}{TQ^2}\frac{C}{q_1} \left(\frac{C^2n_1}{q_1^2}+\frac{Cn_1M_1}{q_1}+{M_{1}^2}\right).
					\end{align*}
					Note that  $$\sup_{N_1 \ll \widetilde{N}} {N_1}N_2=\frac{n_1C^2N_0r}{q_1} T^\epsilon.$$ 
					On using this bound and $N_0= \frac{N}{CQ}$, we get
					\begin{align*}
						\Omega_{\neq 0,\, \mathrm{generic}} \ll \frac{N^{3/2}}{Q^{7/2}T}\frac{ r^2 } {n_{1}^2 q_1}\frac{C^{9/2}}{M_{1}^{1/2}} \left(\frac{C^2n_1}{q_1^2}+\frac{Cn_1M_1}{q_1}+{M_{1}^2}\right).
					\end{align*}
					{\bf Case 2.} $C \ll T^\epsilon QK/T$.\\
					In this case, we will use the   bound  $\mathcal{I}(...)\ll {T^\epsilon C^2N_0}/{(Q^2 T)}$ from Lemma \ref{final int}. Now using this bound and carrying out the same computations as in Case 1, we infer that 
					\begin{align*}
						\Omega_{\neq 0,\, \mathrm{small}} &\ll   \sup_{N_1 \ll \widetilde{N}}\frac{q_{1} {N_1}C r } {n_{1}^3 M_{1}^{1/2}} {N_2}\frac{N_0C^2}{TQ^2}\frac{C}{q_1} \left(\frac{C^2n_1}{q_1^2}+\frac{Cn_1M_1}{q_1}+{M_{1}^2}\right) \\
						&\ll  \frac{\sqrt{N}}{\sqrt{CQ}}\frac{N^{3/2}}{Q^{7/2}T}\frac{ r^2 } {n_{1}^2 q_1}\frac{C^{9/2}}{M_{1}^{1/2}} \left(\frac{C^2n_1}{q_1^2}+\frac{Cn_1M_1}{q_1}+{M_{1}^2}\right).
					\end{align*}
					Hence the lemma follows.
				\end{proof}
				
				\subsection{Final estimates for  $S_{r,\neq 0}(N)$}
				In this subsection, we will estimate  $S_{r,\neq 0}(N)$    using  the bounds for $\Omega_{\neq 0}$ from Lemma \ref{non zero omega bound}. Let  $S_{r,\neq 0, \text{generic}}(N)$ and $S_{r,\neq 0, \text{small}}(N)$   denote the contribution of $\Omega_{\neq 0,\, \mathrm{generic}}$ and $\Omega_{\neq 0,\, \mathrm{small}}$ respectively to  $S_{r,\neq 0}(N)$. Thus, we have
				$$S_{r,\neq 0}(N) \ll |S_{r,\neq 0, \text{generic}}(N)|+|S_{r,\neq 0, \text{small}}(N)|.$$
				Furthermore, let $|\Omega_{\neq 0,\, \mathrm{generic}}^{(j)}|$ denote the part of  $\Omega_{\neq 0,\, \mathrm{generic}}$ corresponding to the $j$-th term in the expression \eqref{omeganon} for $j=1,\,2$ and $3$ and $|S_{r,\neq 0, \text{generic}}^{(j)}(N)|$ denote the corresponding contribution to $|S_{r,\neq 0, \text{generic}}(N)|$. $|S_{r,\neq 0, \text{small}}^{(j)}(N)|$ is defined similarly.  Thus 
				\begin{align}\label{S(N) as large and small}
					S_{r,\neq 0}(N) \ll \sum_{j=1}^3|S_{r,\neq 0, \text{generic}}^{(j)}(N)|+\sum_{j=1}^3|S_{r,\neq 0, \text{small}}^{(j)}(N)|.
				\end{align}
				In the following lemma, we estimate the terms corresponding to $j=2$ and $3$.
				\begin{lemma} \label{genericbou}
					For $j=2$ and $3$, we have
					\begin{equation} \label{genericvalue}
						|S_{r,\neq 0, \text{generic}}^{(j)}(N)|  \ll T^{\epsilon}r^{1/2}N^{3/4} T K^{-1/2} .
					\end{equation}
					and
					\begin{equation} 
						|S_{r,\neq 0, \text{small}}^{(j)}(N)|  \ll T^{\epsilon}r^{1/2}N^{3/4} \left(\frac{K^{5/4}}{T^{1/2}}+\frac{K}{T^{1/4}}\right) \ll T^{\epsilon}r^{1/2}N^{3/4}KT^{-1/4}.
					\end{equation}
				\end{lemma}
				
				\begin{proof} 
					Recall from  \eqref{omeganon} that 
					\begin{align*}
						\Omega_{\neq 0,\, \mathrm{generic}} \ll \frac{N^{3/2}}{Q^{7/2}T}\frac{ r^2 } {n_{1}^2 q_1}\frac{C^{9/2}}{M_{1}^{1/2}} \left(\frac{C^2n_1}{q_1^2}+\frac{Cn_1M_1}{q_1}+{M_{1}^2}\right).
					\end{align*}
					Thus 
					$$\Omega_{\neq 0,\, \mathrm{generic}}^{(3)} \ll \frac{N^{3/2}}{Q^{7/2}T}\frac{ r^2 } {n_{1}^2 q_1}\frac{C^{9/2}}{M_{1}^{1/2}}M_1^2.$$ 
					On substituting it into  \eqref{srsum}, we see  that  $	|S_{r,\neq 0, \text{generic}}^{(3)}(N)| $ is bounded by 
					\begin{align}\label{generic 3}
						&   \sup_{ \substack{C, \,  M_1 
						}}	\ \frac{T^{\epsilon} N^{5/4}}{r^{1/2} KQ C^2}  \sum_{\frac{n_{1}}{(n_{1},r)} \ll C} \Xi^{1/2}  \, \sum_{\frac{n_{1}}{(n_{1},r)} \vert q_{1} \vert (n_{1}r)^{\infty}} \, \sqrt{|\Omega_{\neq 0,\, \mathrm{generic}}^{(3)}|} \notag \\
						&\ll    \sup_{ \substack{C \ll Q
						}}	\ \frac{T^{\epsilon} N^{5/4}}{r^{1/2} KQ C^2} \frac{r N^{3/4} C^{9/4} \widetilde{M}^{3/4} }{Q^{7/4} T^{1/2}}  \sum_{\frac{n_{1}}{(n_{1},r)} \ll C} \Xi^{1/2}  \, \sum_{\frac{n_{1}}{(n_{1},r)} \vert q_{1} \vert (n_{1}r)^{\infty}} \,  \frac{1}{n_1\sqrt{q_1}} \notag \\
						& \ll  \frac{N\sqrt{r}\widetilde{M}^{3/4}}{Q^{3/4}\sqrt{T}} \,{Q^{1/4}}   \sum_{\frac{n_{1}}{(n_{1},r)} \ll C} \Xi^{1/2}\frac{(n_1,r)^{1/2}}{n_1^{3/2}}.
					\end{align}
					We estimate  the sum over $n_1$ (see also \eqref{Xivalue}) as fellows
					\begin{align*} 
						\sum_{n_{1} \ll Cr} \frac{\Xi^{1/2}}{n_1^{3/2}}(n_1,r)^{1/2} \ll \left[\sum_{n_{1} \ll Cr}\frac{(n_1,r)}{n_1}\right]^{1/2}\left[\mathop{\sum \sum}_{n_1^2n_2 \ll \tilde{N}} \frac{\vert \lambda_\pi(n_{1},n_{2})\vert^2}{n_1^2n_{2}}\right]^{1/2} \ll T^{\epsilon}.
					\end{align*}
					Finally, on  plugging $\widetilde{M} \ll Q^2T^2/N$, and $Q=T^\epsilon({N}/{K})^{1/2}$, we arrive at
					\begin{align}\label{genericvalue1234}
						S_{r,\neq 0, \text{generic}}^{(3)}(N) &\ll T^\epsilon N^{3/4} T K^{-1/2} r^{1/2}.
					\end{align}
					Next we estimate  $|S_{r,\neq 0, \text{generic}}^{(2)}(N)|$. On using  
					$$\Omega_{\neq 0,\, \mathrm{generic}}^{(2)}\ll \frac{N^{3/2}}{Q^{7/2}T}\frac{ r^2 } {n_{1}^2 q_1}\frac{C^{9/2}}{M_{1}^{1/2}}\frac{CM_1n_1}{q_1} $$ 
					and proceeding like above, we see  that $|S_{r,\neq 0, \text{generic}}^{(2)}(N)|$ is bounded by
					\begin{align}\label{generic 2}
						&   \sup_{ \substack{C, \,  M_1 
						}}	\ \frac{T^{\epsilon} N^{5/4}}{r^{1/2} KQ C^2} \sum_{\frac{n_{1}}{(n_{1},r)} \ll C} \Xi^{1/2}  \, \sum_{\frac{n_{1}}{(n_{1},r)} \vert q_{1} \vert (n_{1}r)^{\infty}} \, \sqrt{|\Omega_{\neq 0,\, \mathrm{generic}}^{(2)}|} \notag \\
						&\ll \frac{N\sqrt{r}\widetilde{M}^{1/4}}{Q^{3/4}\sqrt{T}}   {Q^{3/4}}   \sum_{\frac{n_{1}}{(n_{1},r)} \ll C} \Xi^{1/2}  \, \sum_{\frac{n_{1}}{(n_{1},r)} \vert q_{1} \vert (n_{1}r)^{\infty}} \,  \frac{1}{\sqrt{n_1}{q_1}}\notag \\
						& \ll  \frac{N\sqrt{r}\widetilde{M}^{1/4}}{Q^{3/4}\sqrt{T}} \,{Q^{3/4}}   \sum_{\frac{n_{1}}{(n_{1},r)} \ll C} \Xi^{1/2}\frac{(n_1,r)}{n_1^{3/2}}.
					\end{align}
					Thus we have
					\begin{align*}
						|S_{r,\neq 0, \text{generic}}^{(2)}(N)| &\ll N^{3/4} T K^{-1/2} r^{1/2} \times \frac{\sqrt{rQ}}{{QT}/\sqrt{N}}. 
					\end{align*}
					Note that 
					$$ \frac{\sqrt{rQ}}{{QT}/\sqrt{N}} =\frac{\sqrt{r}N^{1/4}K^{1/4}}{T} \ll \frac{T^{3/4}K^{1/4}}{T} \ll 1.$$ 
					Hence, we have the first part of the Lemma. A similar analysis can be done to estimate $|S_{r,\neq 0, \text{small}}^{(j)}(N)|$. In fact, note that (see Lemma \ref{non zero omega bound})
					\begin{align*}
						\tilde{\Omega}_{\neq 0,\, \mathrm{small}} \ll \sqrt{N_0}\,\left|  \tilde{\Omega}_{\neq 0,\, \mathrm{generic}}\right|,
					\end{align*}
					where $	\tilde{\Omega}_{\neq 0,\, \mathrm{small}}$  and $ \tilde{\Omega}_{\neq 0,\, \mathrm{generic}}$ are the two upper bounds in Lemma  \ref{non zero omega bound}
					Thus we loose $N^{1/4}/(CQ)^{1/4}$, a priori,   in estimating  $|S_{r,\neq 0, \text{small}}^{(j)}(N)|$ in  comparision with    $|S_{r,\neq 0, \text{generic}}^{(j)}(N)|$. However  we save from the length of $C$ as $C \ll T^\epsilon QK/T$. Hence, following \eqref{generic 3}, we get 
					\begin{align*}
						|S_{r,\neq 0, \text{small}}^{(3)}(N)| &\ll   \frac{N\sqrt{r} T^\epsilon}{Q^{3/4}\sqrt{T}} \,  \sup_{C \ll {QKT^\epsilon}/{T} } \widetilde{M}^{3/4}{C^{1/4}}\times \frac{N^{1/4}}{(CQ)^{1/4}} \\
						&\ll\frac{N^{5/4}\sqrt{r} T^\epsilon}{Q\sqrt{T}}\times K^{3/4}\ll T^\epsilon \sqrt{r}N^{3/4}K^{5/4}T^{-1/2},
					\end{align*}
					as $\widetilde{M}= T^\epsilon\{ C^2T^2/N+K\} \ll T^{2\epsilon }K$. Similarly, following \eqref{generic 2}, we get 
					\begin{align*}
						|S_{r,\neq 0, \text{small}}^{(2)}(N)|&\ll \frac{ T^\epsilon N\sqrt{r}\widetilde{M}^{1/4}}{Q^{3/4}\sqrt{T}} \,  \sup_{C \ll \frac{QKT^{\epsilon}}{T} } {C^{3/4}}\times \frac{N^{1/4}}{(CQ)^{1/4}}\times \sqrt{r}\\
						&\ll T^\epsilon \sqrt{r}N^{3/4}KT^{-1/4},
					\end{align*}
					as $Nr^2 \ll T^{3+\epsilon}$. Hence we have the lemma.
				\end{proof}
				The case of $j=1$ has be be dealt with  differently, as there is a term  $M_1^{1/2}$ appearing in the denominator in the expression \eqref{omeganon}. 
				%
				%
				%
				
				\begin{lemma}\label{s(N) for small}
					We have 
					\begin{align*}
						|S_{r,\neq 0, \text{generic}}^{(1)}(N)| \ll  N^{3/4+\epsilon} T^{1/2},
					\end{align*}
					and
					\begin{equation*}
						S_{r,\neq 0, \text{small}}(N) \ll  r^{1/2} N^{3/4+\epsilon} K^{7/8}{T^{-1/8}}.
					\end{equation*}
				\end{lemma}
				
				\begin{proof}
					Let's  recall that 
					\begin{align}\label{nongen}
						|S_{r,\neq 0, \text{generic}}^{(1)}(N)| &\ll   \sup_{ \substack{C, \,  M_1 
						}}	\ \frac{T^{\epsilon} N^{5/4}}{r^{1/2} KQ C^2}   \sum_{\frac{n_{1}}{(n_{1},r)} \ll C} \Xi^{1/2}  \, \sum_{\frac{n_{1}}{(n_{1},r)} \vert q_{1} \vert (n_{1}r)^{\infty}} \, \sqrt{|\Omega_{\neq 0,\, \mathrm{generic}}^{(1)}|},
					\end{align}
					where 
					\begin{align}
						|\Omega_{\neq 0,\, \mathrm{generic}}^{(1)}| \ll \frac{N^{3/2}}{Q^{7/2}T}\frac{ r^2 } {n_{1}^2 q_1}\frac{C^{9/2}}{M_{1}^{1/2}} \left(\frac{C^2n_1}{q_1^2}\right).
					\end{align}
					Since we are in the generic case, i.e., $C \gg T^{\epsilon}QK/T$, we must have $m \sim M_1 \asymp \widetilde{M}\asymp T^2C^2/N$,  otherwise the integral in \eqref{iprimeinte} will be negligibly small (see Remark \ref{sizeofnx}). Hence, on plugging $T^2C^2/N$ in place of $M_1$ in \eqref{nongen}, we see  that $	|S_{r,\neq 0, \text{generic}}^{(1)}(N)| $ is bounded by
					\begin{align*}
						&  \sup_{ \substack{C, \,  M_1 
						}}	\ \frac{T^{\epsilon} N^{5/4}}{r^{1/2} KQ C^2}    \frac{rN^{3/4 }C^{13/4}}{Q^{7/4} M_1^{1/4}T^{1/2}}   \sum_{\frac{n_{1}}{(n_{1},r)} \ll C} \Xi^{1/2}  \, \sum_{\frac{n_{1}}{(n_{1},r)} \vert q_{1} \vert (n_{1}r)^{\infty}} \, \frac{1}{\sqrt{n_1}q_1^{3/2}} \\
						& \ll   \sup_{ \substack{C
						}}	\ \frac{T^{\epsilon} N^{5/4}}{r^{1/2} KQ}    \frac{rN^{3/4 } N^{1/4}}{Q^{7/4} T^{1/2} T^{1/2}} C^{3/4}   \sum_{\frac{n_{1}}{(n_{1},r)} \ll C}\frac{\Xi^{1/2}}{n_1^2}(n,r)^{3/2}.
					\end{align*}
					Estimating the sum over $n_1$ like before, we get 
					\begin{align*}
						|S_{r,\neq 0, \text{generic}}^{(1)}(N)| \ll  {N^{3/4+\epsilon}}{T^{1/2}},
					\end{align*}
					as $Nr^2 \ll T^{3+\epsilon}$. Next we consider  small $C$ , i.e., $C\ll T^\epsilon QK/T$.
					Note that if 
					$$\frac{C^2n_1}{q_1^2} <\frac{Cn_1M_1}{q_1}  \iff \frac{C}{q_1}<M_1,$$
					then we get back to the previous case corresponding to $j=2$. Therefore we can assume that $M_1 <C/q_1$. Furthermore, if $M_{1} \asymp {T^{2}C^{2}}/{N}$,  then in this case also, we are in a similar situation as above. Thus we are left with the case $$ {T^{2}C^{2}}/{N}  \not\asymp M_1 <{C}/{q_1}.$$
					In this case,  We have 
					\begin{equation*} 
						n_{2} \ll  T^\epsilon \frac{n_{1} C^{2} N_{0} r}{q_{1} {N_1}}.
					\end{equation*}
					Note that we can replace $N_1$ by $\widetilde{N}$, as later we take  supremum  over $N_1$ while estimating  $\Omega$ and there is a $N_1$ factor present in the numerator (see \eqref{omegavalue}). Thus 
					\begin{equation} \label{newn2range}
						n_{2} \ll  T^\epsilon\frac{n_{1} C^{2} N_{0} r}{q_{1} {\widetilde{N}}}  \ll \frac{n_{1}}{q_{1}} \frac{N}{Q T^2} T^{\epsilon}.
					\end{equation}
					Here we have used $T^\epsilon N_0 \gg N|x|/CQ \gg T^{1-\epsilon}$, which follows from  Remark \ref{sizeofnx}, and 
					$$\widetilde{N}= \frac{rQ^3K^2T^{1+\epsilon}}{N}\asymp \frac{rQ C^{2}N_{0}^{2} T^{1+\epsilon}}{N}.$$
					from \eqref{paravlues}. We recall  from $\eqref{omega for small m}$ that
					\begin{align*}
						\Omega_{\neq 0} \ll \frac{q_{1}^2 \widetilde{N} r T^{\epsilon}} {n_{1}^3 M_{1}^{1/2}} \mathop{\sum \sum}_{d_{2}, d_{2}^{\prime} \leq C/q_{1} } d_{2} d_{2}^{\prime} \, \mathop{\sum \sum }_{\substack{q_{2} \sim {C}/{d_{2}q_{1}} \\ q_{2}^{\prime} \sim {C}/{d_{2}^{\prime} q_{1}}}} \mathop{ \mathop{\sum \  \sum  \ \ \sum}_{\substack{m,m^{\prime} \sim M_{1}  n_2 \in \mathbb{Z}-\{0\} \\n_{1} q_{2}^{\prime}d_2^\prime - m n_{2} \, \equiv \,  0 \, \mathrm{mod} \, d_{2} \\  n_{1} q_{2}d_2+ m^{\prime} n_{2} \, \equiv \,  0 \, \mathrm{mod} \, d_{2}^{\prime} }}} \, (m,n_1)\vert \mathcal{I}(...)\vert.
					\end{align*}
					Here  we will adopt a different strategy to count $m$ and $m^\prime$. Let
					$$d_{2}^{\prime} \sim D^{\prime}, d_{2} \sim D$$
					with $D^{\prime} \ll D$ and $D, D^{\prime} \leq C/q_{1}$. Now we rewrite the first congruence equation as
					$$n_{1} q_{2}^{\prime} d_{2}^{\prime}-mn_{2} = \ell d_{2},$$
					with
					$$\ell \ll \left(\frac{Cn_{1}}{q_{1} D} + \frac{n_{1}CN}{q_{1}^2 T^{2}Q D} \right) T^{\epsilon} \ll \frac{Cn_1}{q_1D}=: L,$$
					where we used  $m \sim M_{1} \leq C/q_{1}$ and $n_{2} \ll T^{\epsilon}{n_{1}N}/{q_{1} QT^2}$.  The second inequality follows fron the fact $q_1\geq n_1/r$ and $Nr^2 \ll T^{3+\epsilon}$. Hence we arrive at
					\begin{align*}
						\frac{q_{1}^2 \tilde{N} r N_0 C^2} {n_{1}^3 M_{1}^{1/2}TQ^2} \mathop{\sum \sum}_{d_{2} \sim D, d_{2}^{\prime}\sim D^\prime } d_{2} d_{2}^{\prime} \, \mathop{\sum \sum }_{\substack{q_{2} \sim {C}/{d_{2}q_{1}} \\ \ell \ll L}} \mathop{ \mathop{\sum \  \sum  \ \ \sum}_{\substack{m,m^{\prime} \sim M_{1}  0\neq  n_2\ll N_2  \\ \ell d_2 + m n_{2} \, \equiv \,  0 \, \mathrm{mod} \, d_{2}^\prime \\  n_{1} q_{2}d_2+ m^{\prime} n_{2} \, \equiv \,  0 \, \mathrm{mod} \, d_{2}^{\prime} }}} \, (m,n_1).
					\end{align*}
					Here we used the bound $ \mathcal{I}(...) \ll T^\epsilon C^2N_0/ (Q^2T)$. Using the second congruence equation, the number of $m^{\prime}$  turns out to be $$O\left((n_{2},d_{2}^{\prime}) (1+{M_{1}}/{D^{\prime}})\right).$$
					The first congruence equation provides either the count for $d_{2}$ which is  $O((d_{2}^{\prime},\ell) D/D^{\prime})$ if $\ell \neq 0$ or the count for $d_{2}^{\prime}$ which turns out to be $O(T^{\epsilon})$ for $\ell =0$. Thus, taking this into account, we arrive at
					$$ 	\frac{q_{1}^2 \widetilde{N} r N_0 C^2} {n_{1}^3 M_{1}^{1/2}TQ^2} \mathop{ \sum}_{ d_{2}^{\prime}\sim D^\prime }  D^2 \, \mathop{\sum \sum }_{\substack{q_{2} \sim {C}/{Dq_{1}} \\ 0\neq \ell \ll L}}  \sum_{m \sim M_{1}} \sum_{0 \neq n_2\ll N_2}(n_{2},d_{2}^{\prime}) (\ell, d_{2}^{\prime}) (m,n_{1}) \left(1+\frac{M_{1}}{D^\prime}\right).$$
					First summing over $m$, and then over $\ell, n_{2},\,q_{2}$ and $ d_{2}^{\prime}$, we arrive at
					\begin{align*}
						&	\frac{T^\epsilon q_{1}^2 \widetilde{N} r N_0C^2} {n_{1}^3 M_{1}^{1/2}T Q^2}D^2\frac{C}{Dq_1}M_1LN_2D^{\prime}\left(1+\frac{M_{1}}{D^\prime}\right) \\
						& \ll 	\frac{T^\epsilon q_{1}^2 \widetilde{N} r N_0 C^2} {n_{1}^3 TQ^2}\frac{C}{q_1}\sqrt{M_1}N_2\left(D^\prime+{M_{1}}\right)\left(\frac{Cn_{1}}{q_{1} }  \right) \\
						& \ll 	\frac{ T^\epsilon q_1^2 r\widetilde{N} N_2 N_0 C^2} {n_{1}^3 TQ^2}\frac{C}{q_1}\left(\frac{C}{q_1}\right)^{1/2}\frac{C}{q_1}\left(\frac{Cn_{1}}{q_{1}  } \right).
					\end{align*}
					Now using the  bound $$\widetilde{N}N_2N_0\asymp \frac{n_1r}{q_1}C^2N_0^2\asymp \frac{n_1r}{q_1}KN,$$
					we get 
					\begin{align*}
						\Omega_{\neq 0,\, \mathrm{small}} \ll \frac{C^{11/2}KN}{TQ^2}\frac{r^2}{n_1q_1^{5/2}}.
					\end{align*}
					On substituting this bound in \eqref{srsum}, we  see  that $	S_{r,\neq 0, \mathrm{small}}(N) $ is bounded by 
					\begin{align*}
						& 	\sup_{ \substack{C, \,  M_1 
						}}	\ \frac{T^{\epsilon} N^{5/4}}{r^{1/2} KQ C^2}  \sum_{\frac{n_{1}}{(n_{1},r)} \ll C} \Xi^{1/2}  \, \sum_{\frac{n_{1}}{(n_{1},r)} \vert q_{1} \vert (n_{1}r)^{\infty}} \, \left(\frac{C^{11/2}KN}{TQ^2}\frac{r^2}{n_1q_1^{5/2}} \right)^{1/2} \\
						&\ll \sup_{C \ll  T^\epsilon QK/T} \frac{N^{1/4}C^{3/4}\sqrt{rKN}}{  \sqrt{T}}\sum_{\frac{n_{1}}{(n_{1},r)} \ll C} \Xi^{1/2}  \, \sum_{\frac{n_{1}}{(n_{1},r)} \vert q_{1} \vert (n_{1}r)^{\infty}} \, \left(\frac{1}{n_1^{1/2}q_1^{5/4}} \right) \\
						& \ll \frac{N^{1/4}Q^{3/4}K^{3/4}\sqrt{rKN}}{  T^{3/4}\sqrt{T}}\sum_{\frac{n_{1}}{(n_{1},r)} \ll C} \frac{\Xi^{1/2}}{n_1^{7/4}}(n_1,r)^{5/4} 
					\end{align*}
					Estimating the sum over $n_1$ like before, we have 
					\begin{align*} 
						\sum_{n_{1} \ll Cr} \frac{\Xi^{1/2}}{n_1^{7/4}}(n_1,r)^{5/4} \ll \left[\sum_{n_{1} \ll Cr}\frac{(n_1,r)^{5/2}}{n_1^{3/2}}\right]^{1/2}\left[\mathop{\sum \sum}_{n_1^2n_2 \ll \tilde{N}} \frac{\vert A(n_{1},n_{2})\vert^2}{n_1^2n_{2}}\right]^{1/2} \ll \sqrt{r} T^{\epsilon}.
					\end{align*}
					Hence, we obtain
					$$S_{r,\neq 0, \mathrm{small}}(N) \ll \frac{T^\epsilon rN^{3/4}Q^{3/4}K^{5/4}}{T^{5/4}}\asymp  \frac{T^\epsilon rN^{3/4}N^{3/8}K^{7/8}}{T^{5/4}}\ll  \frac{T^\epsilon r^{1/2} N^{3/4}K^{7/8}}{T^{1/8}}.$$
					Hence the lemma follows.
				\end{proof}
				We conclude this section by combining all   the above lemmas into the following lemma.
				\begin{lemma}\label{final non zero S(N) bound}
					We have 
					\begin{align*}
						S_{r,\neq 0, \text{generic}}(N) \ll  r^{1/2} N^{3/4+\epsilon }TK^{-1/2},
					\end{align*}
					and
					\begin{align*}
						S_{r,\neq 0, \text{small}}(N)  \ll r^{1/2} N^{3/4+\epsilon } K^{7/8}{T^{-1/8}},
					\end{align*}
					and therefore, 
					\begin{align*}
						S_{r, \neq 0}(N) \ll T^{\epsilon}r^{1/2} N^{3/4}\left(\frac{T}{\sqrt{K}}+\frac{K^{7/8}}{T^{1/8}}\right).
					\end{align*}
				\end{lemma}
				\begin{proof}
					The proof follows by plugging bounds from Lemma \ref{genericbou} and Lemma \ref{s(N) for small} into \eqref{S(N) as large and small}.
				\end{proof}

				\section{Proof of Theorem \ref{maintheorem}}
				In this section, we will prove Theorem \ref{maintheorem}  using Lemmas proved so far. We recall from Lemma \ref{AFE}  that
				\begin{align}\label{L function in term of S(N)}
					L(\phi \times f,1/2) \ll_{\epsilon} T^{\epsilon}\sup_{r \leq T^{(3+\epsilon)/2}} \sup_{  Nr^2\leq {T^{3+\epsilon}}} \frac{\left|S_r(N)\right|}{N^{1/2}} + T^{-2020},
				\end{align}
				On applying  Lemma \ref{zerofrequencybound} and Lemma \ref{final non zero S(N) bound}  to   
				$$|S_{r}(N)| \ll |S_{r,0}(N)|+ |S_{r,\neq 0}(N)|,$$
				we see that 
				\begin{align*}
					{S_{r}(N)}  &\ll \sqrt{r}N^{3/4} \left(T^{1/2} K^{1/4} +  T K^{-1/2}+K^{7/8}T^{-1/8}\right) \\
					&\ll \sqrt{r}N^{3/4} \left(T^{1/2} K^{1/4} +  T K^{-1/2}\right)\\
					&\ll \sqrt{r}N^{3/4}\left(T^{3/4-\xi/4} +  T^{1/2+\xi/2} \right) ,
				\end{align*}
				as $K=T^{1 -\xi} <T$ for some $0 < \xi < 1$. Hence, using $N r^2 \ll T^{3+\epsilon}$,  we get
				$$\frac{S_{r}(N)}{\sqrt{N}} \ll \sqrt{r}N^{1/4}\left(T^{3/4-\xi/4} +  T^{1/2+\xi/2} \right)\ll \max \left\{T^{\frac{3}{2}-\frac{\xi}{4}} , T^{\frac{3}{2}-\frac{1-2 \xi}{4}} \right\} T^{\epsilon}.$$
				Finally  on plugging this  bound into \eqref{L function in term of S(N)} we infer that
				$$L(\phi \times f, {1}/{2}) \ll \max \left\{ T^{\frac{3}{2}-\frac{\xi}{4}+\epsilon} ,\, T^{\frac{3}{2}-\frac{1-2 \xi}{4}+\epsilon} \right\}.$$
				We note that the above bound is sub-convex whenever  $0 <\xi <1/2$. 
				Thus we conclude the proof of Theorem \ref{maintheorem}.

				\section*{Acknowledgements} 
				We are thankful to Prof. Ritabrata Munshi for  explaining  his method and his constant support throughout the work. We also  thank Prof. Satadal Ganguly for his encouragement.  We are thankful to Prof. Gergely Harcos for valuable suggestions.   Authors are grateful to  Stat-Math Unit, Indian Statistical Institute, Kolkata  where most of the work was done.  Part of the work was done at Erd\"os center, Alfr\'ed R\'enyi Institute of Mathematics, IIT Bombay and IIT Kanpur. We are thankful to these institute for    providing wonderful research environment.

			\end{document}